\documentclass[12pt,reqno]{amsart}
\usepackage{latexsym,amsmath,amsfonts,amscd,amssymb}
\usepackage{graphics}

\usepackage{float}
\usepackage{tikz}
\usepackage{dsfont}
\usetikzlibrary{matrix}

\usepackage{tikz-cd}

\newcommand{\cota}{\tau}

\newcommand{\la}{\langle}
\newcommand{\ra}{\rangle}
\newcommand{\x}{\times}

\newcommand{\bd}{\partial}
\newcommand{\cA}{{\mathcal{A}}}
\newcommand{\cC}{{\mathcal{C}}}
\newcommand{\cD}{{\mathcal{D}}}

\newcommand{\cP}{{\mathcal{P}}}
\newcommand{\ZZ}{\mathbb{Z}}
\newcommand{\CC}{\mathbb{C}}
\newcommand{\CP}{\mathbb{C}P}

\newcommand{\RR}{\mathbb{R}}

\newcommand{\QQ}{\mathbb{Q}}

\DeclareMathOperator{\lcm}{lcm}

\DeclareMathOperator{\Id}{Id}

\newcommand{\s}{\sigma}

\newcommand{\cO}{\mathcal O}

\newcommand{\orb}{\scriptsize\mathrm{orb}}

\newtheorem{theorem}{Theorem}
\newtheorem{proposition}[theorem]{Proposition}
\newtheorem{lemma}[theorem]{Lemma}
\newtheorem{definition}[theorem]{Definition}

\newtheorem{corollary}[theorem]{Corollary}
\newtheorem{remark}[theorem]{Remark}

\newtheorem{conditions}[theorem]{Conditions}

\title[A Smale-Barden manifold admitting K-contact but not Sasakian]{A Smale-Barden manifold admitting K-contact but not Sasakian structure}

\author[V. Mu\~{n}oz]{Vicente Mu\~{n}oz}
\address{Instituto de Matem\'atica Interdisciplinar and
Departamento de \'Algebra, Geometr\'{\i}a y Topolog\'{\i}a, Universidad Complutense de Madrid, Plaza de las Ciencias, 3, 28040-Madrid, Spain}
\email{vicente.munoz@ucm.es}



\thanks{Partially supported by Project MINECO (Spain) PID2020-118452GB-I00}

\begin{document}

\begin{abstract}
We give the first example of a simply connected compact $5$-manifold (Smale-Barden manifold) which
admits a K-contact structure but does not admit any Sasakian structure, settling a long standing
question of Boyer and Galicki.
\end{abstract}

\maketitle

\section{Introduction}\label{sec:intro}

In geometry, a central question is to determine when a given manifold admits a specific geometric structure. 
Complex geometry provides numerous examples of compact manifolds with rich topology, and there is a number
of topological properties that are satisfied by K\"ahler manifolds \cite{ABCKT, DGMS}.
If we forget about the integrability of the complex structure, then we are dealing with symplectic manifolds. There has
been enormous interest in the construction of (compact) symplectic manifolds that do not admit K\"ahler structures,
and in determining its topological properties \cite{OT}. The fundamental group is one of the more direct invariants
that constrain the topology of K\"ahler manifolds \cite{DGMS}, whereas any finitely presented group can be the 
fundamental group of a compact symplectic manifold \cite{Gompf}. For this reason, the problem becomes more
relevant if we ask for simply connected compact manifolds. On the other hand, the difficulties increase as we
look for manifolds of the lowest possible dimension. For instance, the lowest dimension for a 
compact simply connected manifold admitting a symplectic but not a K\"ahler structure and having \emph{non-formal}
rational homotopy type is $8$. Such an example was first provided by Fern\'andez and the author 
in \cite{FM-annals}. Also, a compact simply connected
manifold admitting both a symplectic and a complex structure but not a K\"ahler structure can only happen in dimensions
higher than $6$. The first example of such instance in the lowest dimension $6$ is given by Bazzoni, Fern\'andez
and the author in \cite{BFM-JSG}.

A symplectic manifold always admits an almost-K\"ahler structure (which is a metric structure), so the topological question
above can be rephrased as to finding manifolds which admit almost-K\"ahler but no K\"ahler structures.
In odd dimensions, the analogues of K\"ahler and almost-K\"ahler manifolds are Sasakian and K-contact manifolds, respectively
(and the analogue of symplectic manifold is contact manifold). These are metric structures which are endowed with 
a one-dimensional foliation and a transversal structure which is K\"ahler or almost-K\"ahler, respectively (Section \ref{sec:k-contact} for
precise definitions). 
Sasakian geometry has become an important and active subject since the treatise of Boyer and Galicki \cite{BG}, 
and there is much interest on constructing K-contact manifolds which do not admit Sasakian structures.
As mentioned in \cite[Charpter 10]{BG}, now there is a gap between contact and K-contact, and the problem
of finding a manifold admitting a contact but not Sasakian structure is easily solved. However, finding manifolds which admit
a K-contact but not a Sasakian structure is harder.

The parity of $b_1$ was used to produce the first examples of K-contact manifolds with no Sasakian structure
\cite[example 7.4.16]{BG}. More refined tools are needed in the case of even Betti numbers. The 
cohomology algebra of a Sasakian manifold satisfies a hard Lefschetz property \cite{CNY}, and using
it examples of K-contact non-Sasakian manifolds are produced in \cite{CNMY} in dimensions $5$ and $7$. These
examples are nilmanifolds with even Betti numbers, so in particular they are not simply connected.
The fundamental group can also be used to construct K-contact non-Sasakian manifolds \cite{Chen}. 
Also it has been used to provide an example of a solvmanifold of dimension $5$ which satisfies the hard Lefschetz 
property and which is K-contact and not Sasakian \cite{CNMY2}.

When one moves to the case of simply connected manifolds, K-contact non-Sasakian 
examples of any dimension $\geq 9$ were constructed in \cite{HT} using the evenness of the third Betti
number of a compact Sasakian manifold. Alternatively, using the hard Lefschetz property
for Sasakian manifolds there are examples \cite{Lin} of simply connected K-contact non-Sasakian
manifolds of any dimension $\geq 9$.
In \cite{BFMT} the rational homotopy type is used to construct examples of simply connected 
K-contact non-Sasakian manifolds in dimensions $\geq 17$. In dimension $7$ there are 
examples in \cite{MT} of simply connected K-contact non-Sasakian manifolds. 
However, Massey products are not suitable for the analysis of lower dimensional manifolds.
The problem of the existence of simply connected K-contact non-Sasakian compact manifolds 
is still open in dimension $5$, despite numerous attempts. 

\medskip
{\centerline{
\begin{minipage}{14cm}
\noindent Open Problem 10.2.1 in \cite{BG}:
{\it Do there exist Smale-Barden manifolds which carry K-contact but do not carry Sasakian structures?}
\end{minipage}}
} \medskip

A simply connected compact $5$-manifold is called a {\it Smale-Barden manifold}. 
These manifolds are classified \cite{B,S} by $H_2(M,\mathbb{Z})$ and the second Stiefel-Whitney class
(see Section \ref{sec:SB}). This makes sensible to pose classification problems of manifolds admiting diverse geometric structures
in the class of Smale-Barden manifolds.

A Sasakian manifold $M$ always admits a \emph{quasi-regular} Sasakian structure. This
gives $M$ the structure of a Seifert bundle over a cyclic K\"ahler orbifold $X$ (see Section \ref{subsec:orbifold}). 
In the case of a $5$-manifold, $X$ is a singular
complex surface with cyclic quotient singularities. The Sasakian structure is \emph{semi-regular} if the isotropy locus is only formed
by codimension $2$ submanifolds, that is if $X$ is a smooth complex surface and the isotropy consists of smooth complex
curves (maybe intersecting). A similar statement holds for K-contact manifolds, where the base is now an almost-K\"ahler orbifold
(that is, symplectic with a compatible almost complex structure, which always exists) with cyclic singularities, and the isotropy locus
is formed by symplectic surfaces.

In \cite{K} Koll\'ar determines the topology of simply connected $5$-manifolds which are Seifert bundles over semi-regular $4$-orbifolds.
The torsion in $H_2(M,\ZZ)$ is determined by the genera and isotropy coefficients of the isotropy surfaces.
He uses this to produce simply connected $5$-manifolds which are Seifert bundles (that is, they admit a fixed point free circle action)
but which do not admit a Sasakian structure. If the structure is semi-regular, the isotropy surfaces must satisfy the adjunction equality,
so an example violating it will produce such example. In general, a K\"ahler orbifold can have isolated singularities which cause
serious difficulties, since the classification of singular complex surfaces is far more complicated than that of smooth complex surfaces.
Koll\'ar uses the case $b_2(X)=1$, where there is a bound on the number of singular points, and taking enough curves for the 
isotropy locus ensures that some of them satisfy the adjunction equality.

 To produce K-contact Smale-Barden manifolds, one needs to construct symplectic $4$-manifolds (or $4$-orbifolds with cyclic quotient
 singularities) with symplectic surfaces of given genus inside. If the isotropy coefficients are not coprime, these surfaces are forced to be
 disjoint (and linearly independent in homology). Therefore there is a bound on the number $k$ of surfaces in the isotropy locus,
 $k\leq b_2(X)$. The genus of the isotropy surfaces, the isotropy coefficients, and whether they are disjoint, are translated to the 
 homology group
 $H_2(M,\ZZ)$ of the $5$-manifold $M$. This is used in \cite{MRT} to produce a homology Smale-Barden manifold (that is, a $5$-manifold $M$ with 
$H_1(M,\ZZ)=0$ instead of simply connected) which admits a semi-regular K-contact but not a semi-regular Sasakian structure.
For this, we construct a simply connected symplectic $4$-manifold $X$ with $k$ \emph{disjoint} symplectic surfaces of positive genus
where $k=b_2(X)$, and linearly independent in homology. To prove that this is not semi-regular Sasakian, we have to check that 
there is no complex surface $Y$ with $b_1(Y)=0$, and $k$ disjoint complex curves of positive genus where $k=b_2(Y)>1$, 
and linearly independent in homology, at least for the case where the genera match our symplectic example. 
The existence of so many disjoint complex curves of positive genus and generating the rational homology  
is certainly a rare phenomenon and we conjecture that it does never happen. Unfortunately, as the example in \cite[Section 3]{CMST} 
shows, this can happen for singular complex manifolds with cyclic singularities. For this reason, we do not know whether the example
in \cite{MRT} can admit a quasi-regular Sasakian structure. 

Later, in \cite{CMRV} we extend the ideas of \cite{MRT} to produce the first example of \emph{simply connected} $5$-manifold which
admits a semi-regular K-contact structure but not a semi-regular Sasakian structure. Again we have not been able to remove the
semi-regularity assumption. The purpose of this paper is to completely settle the question in \cite[Open Problem 10.2.1]{BG}.

\begin{theorem} \label{thm:main}
There exists a Smale-Barden manifold $M$ which admits a K-contact structure but does not admit a Sasakian structure.
\end{theorem}

More precisely, a manifold $M$ in Theorem \ref{thm:main} can be explicitly given as follows. There is some 
$N>0$ large enough, and distinct primes $p_{nm}$ with $p_{nm}>\max(3,n,m)$, $1\leq n,m\leq N$, so that 
$M$ is the Smale-Barden manifold characterized  by the fact $M$ is spin and its homology is
$$
  H_2(M,\ZZ)= \ZZ^2 \oplus \bigoplus_{n,m=1}^N 
 \left(\ZZ^{18n^2+2}_{p_{nm}}  \oplus \ZZ^{18m^2+2}_{p_{nm}^2} \oplus \ZZ^{20}_{p_{nm}^3}\right) .
 $$

Let us describe the philosophy behind the construction in Theorem \ref{thm:main}, although the technical details, which
will be carried out in the following sections, get quite involved. 
As we said before, for our Seifert bundles $M\to X$ we can keep track of the genera and the disjointness of the 
isotropy surfaces. With this we try to push the construction to record also the value of $b^+_2(X)$. We note that
when $Y$ is a complex surface with $k=b_2(Y)$ complex curves spanning $H_2(Y,\QQ)$, then $b^+_2(Y)=1$, whereas
the same property does not necessarily hold for symplectic $4$-manifolds. If $X$ is symplectic and $b^+_2(X)>1$, then
when we have $k=b_2(X)$ disjoint symplectic surfaces, we can take positive multiples of those with positive self-intersection.
This gives $N$ families of $k=b_2(X)$ disjoint symplectic surfaces, which can be used as isotropy locus, for any $N\gg 0$ as
large as we want.

The proof that the resulting $5$-manifold $M$ does not admit a Sasakian structure now requires to check that
there is no singular complex surface $Y$ with cyclic singularities with a large number of families, each consisting 
of $k=b_2(Y)$ disjoint complex curves.
First we need to bound the number of singular points (universally, i.e.\ independently of $Y$) as it was done in \cite{K}
for the easy case $b_2(Y)=1$. This serves to bound geometric quantities, like the Euler characteristic, $K^2$,
or the self-intersection of negative curves. For orbifolds, the intersection and self-intersection numbers and $K^2$ can
be rational (instead of just integers), so it is necessary to bound the denominators (independently of $Y$).

As the number of singular points in bounded, we have that 
most of the families of disjoint complex curves avoid the singular points. However, now the genera of 
the curves have increased (the arguments of \cite{CMRV, MRT} deal with cases of low genus
curves, so they are not helpful now). 
In the families of disjoint complex curves, it cannot happen that the curves are multiples of $k$ fixed curves (incidentally, note that this was
the way in which the symplectic example $X$ is produced), because that would imply that $b^+_2(Y)>1$, and this does not
hold for an algebraic surface. This forces to have from the initial $N$ families of $k$ disjoint complex curves 
(these are orthogonal bases of $H_2(Y,\QQ)$), many of them whose elements 
are not proportional to each other (what we call proj-equivalent bases). 
The final step is to prove the impossibility of this situation,
by writing $K^2$ with respect to each of the orthogonal basis, 
and use the bounds on the denominators of the rational numbers. We get a collection of
diophantine equalities, and choosing $N$ large enough, these become incompatible.

\subsection*{Acknowledgements} I am grateful to Javier Fern\'andez de Bobadilla, Marco Castrill\'on, Antonio Viruel and Maribel
Gonz\'alez-Vasco for encouragement. Also thanks to Matthias Sch\"utt and Alex Tralle for 
very useful conversations, and to Sergio Negrete for help with designing a program to find elliptic
surfaces for Section \ref{sec:k-cont}. Special thanks to Juan Rojo for carefully reading the manuscript and finding
a mistake in a previous version, and to \'Angel Gonz\'alez-Prieto for help with drawing the pictures. 
Finally, I am grateful to the referees for very helpful comments that have improved the
exposition, and for the kind words of praise.

\section{Basic notions}

\subsection{Smale-Barden manifolds} \label{sec:SB}

A $5$-dimensional simply connected manifold is called a {\it Smale-Barden manifold}. 
These manifolds are classified by their second homology group over $\ZZ$ and the Barden invariant \cite{B,S}. In more detail,
let $M$ be a compact smooth simply connected $5$-manifold and 
write $H_2(M,\ZZ)$ as a direct sum of cyclic groups of prime  power order
  \begin{equation}\label{eqn:H2Z}
  H_2(M,\ZZ)=\ZZ^k\oplus \left( \mathop{\oplus}\limits_{p,i}\ZZ_{p^i}^{c(p^i)}\right),
  \end{equation}
where $k=b_2(M)$. The equality (\ref{eqn:H2Z}) is actually an isomorphism as abelian groups.
We can arrange so that the second Stiefel-Whitney class map
  $$
  w_2: H_2(M,\ZZ)\rightarrow\ZZ_2
  $$
is zero on all but one summand (or zero on all if $w_2=0$). For that, take an element
on which it is not zero, and complete to a generating system with elements in the kernel.
If $w_2$ is non-zero on $\ZZ_{2^j}$, we set $i(M)=j$; if $w_2$ is non-zero on a summand
$\ZZ$, we set $i(M)=\infty$; if $w_2=0$ then we set $i(M)=0$.
The number $i(M)$ is called the Barden invariant and determines $w_2$ up to isomorphism
of abelian groups.
A Smale-Barden manifold $M$ is uniquely characterized by its homology (\ref{eqn:H2Z}) and $i(M)$.

We shall not use the following, but we include for completeness. The geometric description of 
Smale-Barden manifolds corresponding to these abelian groups is given as follows.

\begin{theorem}[{\cite[Theorem 10.2.3]{BG}}]\label{thm:sb-classification} 
Any simply connected closed $5$-manifold is diffeomorphic to one of the spaces
 $$
  M_{j;k_1,...,k_s;r}=X_j\# r M_\infty \#M_{k_1}\#\cdots\#M_{k_s} 
   $$
where the manifolds $X_{-1},X_0,X_j,X_{\infty}, M_j,M_{\infty}$ are characterized as follows: $1<k_i< \infty$, $k_1|k_2|\ldots |k_s$, and
\begin{itemize}
\item $X_{-1}=SU(3)/SO(3)$, $H_2(X_{-1},\ZZ)=\ZZ_2$, $i(X_{-1})=1$,
\item $X_0=S^5$, $H_2(X_0,\ZZ)=0$, $i(X_0)=0$,
\item $X_j$, $0<j<\infty$, $H_2(X_j,\ZZ)=\ZZ_{2^j}\oplus \ZZ_{2^j}$, $i(X_j)=j$,
\item $X_{\infty}=S^2\widetilde\times S^3$, the unique non-trivial $S^3$-bundle over $S^2$, $H_2(X_{\infty},\ZZ)=\ZZ$, $i(X_\infty)=\infty$,
\item $M_k$, $1<k<\infty$, $H_2(M_k,\ZZ)=\ZZ_k\oplus \ZZ_k$, $i(M_k)=0$,
\item $M_{\infty}=S^2\times S^3$, $H_2(M_{\infty},\ZZ)=\ZZ$, $i(M_{\infty})=0$.
\end{itemize}
\end{theorem}

\subsection{Sasakian and K-contact manifolds}\label{sec:k-contact}

Let $(M,\eta)$ be a co-oriented contact manifold with a contact form
$\eta\in \Omega^1(M)$, i.e.\ $\eta\wedge (d\eta)^n>0$ everywhere, with $\dim M=2n+1$. The 
manifold $M$ is automatically oriented.
We say that $(M,\eta)$ is {\em K-contact} if there is an
endomorphism $\Phi$ of $TM$ such that:
 \begin{itemize}
\item $\Phi^2=-\Id + \xi\otimes\eta$, where $\xi$ is the Reeb
vector field of $\eta$ (that is $i_\xi \eta=1$, $i_\xi (d\eta)=0$),
\item the contact form $\eta$ is compatible with $\Phi$ in the sense that
$d\eta (\Phi X,\Phi Y)\,=\,d\eta (X,Y)$, for all vector fields $X,Y$,
\item $d\eta (\Phi X,X)>0$ for all nonzero $X\in \ker \eta$, and
\item the Reeb field $\xi$ is Killing with respect to the
Riemannian metric defined by the formula
 $g(X,Y)\,=\,d\eta (\Phi X,Y)+\eta (X)\eta(Y)$.
\end{itemize}
In other words, the endomorphism $\Phi$ defines a
complex structure on $\cD=\ker \eta$ compatible with $d\eta$, hence
$\Phi$ is orthogonal with respect to the metric
$g|_\cD$. By definition, the Reeb vector
field $\xi$ is orthogonal to $\cD$, and it is a Killing vector field.

Let $(M,\eta,\xi,\Phi,g)$ be a K-contact manifold. Consider the contact cone as the Riemannian manifold
$C(M)=(M\times\mathbb{R}_+, t^2g+dt^2)$.
One defines the almost complex structure $I$ on $C(M)$ by:
 \begin{itemize}
\item $I(X)=\Phi(X)$ on $\ker\eta$,
\item $I(\xi)=t{\partial\over\partial t},\,I(t{\partial\over\partial t})=-\xi$, for the Killing vector field $\xi$ of $\eta$.
\end{itemize}
We say that $(M,\eta,\xi,\Phi,g)$  is  {\it Sasakian} if $I$ is integrable.
Thus, by definition, any Sasakian manifold is K-contact.

The Sasakian structure can also be defined by the 
integrability of the almost contact metric structure. More precisely, an almost contact
metric structure $(\eta , \xi , \Phi , g)$ is called \emph{normal} if the Nijenhuis
tensor $N_{\Phi}$ associated to the tensor field $\Phi$, defined by
 $$
N_{\Phi} (X, Y) \,:=\, {\Phi}^2 [X, Y] + [\Phi X, \Phi Y] -  \Phi [ \Phi X, Y] - \Phi [X, \Phi Y]\, ,
 $$
satisfies the equation
$$
N_{\Phi} \,=\,-{d}\eta\otimes \xi\, .
$$
Then a {Sasakian structure} is a normal contact metric structure.

A Sasakian (compact) manifold $M$ has a $1$-dimensional foliation defined by the Reeb vector field, which gives an isometric flow,
and the transversal structure is K\"ahler. The Sasakian structure is called \emph{quasi-regular} if the leaves of the Reeb flow are
circles, in which case the leaf space $X$ is a K\"ahler cyclic orbifold and
the quotient map $\pi:M\to X$ has the structure of a Seifert bundle \cite[Theorem 7.13]{BG}. Remarkably, a manifold $M$ admitting
a Sasakian structure also has a quasi-regular one \cite{R}. So from the point of view of whether $M$ admits a Sasakian 
structure, we can assume that it is a Seifert bundle over a K\"ahler cyclic orbifold. The Sasakian structure
is \emph{regular} if $X$ is a K\"ahler manifold (no isotropy locus), and \emph{semi-regular} if the isotropy locus
has only codimension $2$ strata (maybe intersecting), or equivalently if $X$ has underlying space which is a topological manifold.

In the case of a K-contact manifold, the situation is analogous, with the difference that 
the transversal structure is almost-K\"ahler.
We define regular, quasi-regular and semi-regular K-contact structures
with the same conditions. Any K-contact manifold admits a quasi-regular K-contact structure  \cite{MT}, and
hence a K-contact manifold is a Seifert circle bundle over a symplectic cyclic orbifold. In the
case $\dim M=5$, $\dim X=4$, such orbifold has isotropy locus
which is a collection of symplectic surfaces and points. 
From the point of view of whether a manifold admits a K-contact structure, we can always
assume that the manifold is a Seifert bundle over a symplectic cyclic orbifold.

\subsection{Cyclic orbifolds}\label{subsec:orbifold}
Let $X$ be a $4$-dimensional (oriented) cyclic orbifold. For the notions about orbifolds 
the reader can consult \cite{BFMT, BG,Mu}.
Let $x\in X$ be a point. A neighbourhood of $x$ is an open subset $U=\tilde U/\ZZ_m$,
where $\tilde U \subset \CC^2$ and 
the action of $\ZZ_m=\la \varepsilon\ra$, $\varepsilon=e^{2\pi i /m}$, is given by
 \begin{equation}\label{eqn:action}
 \varepsilon\cdot (z_1,z_2)=(\varepsilon^{j_2} z_1, \varepsilon^{j_1} z_2),
 \end{equation}
where $j_1,j_2$ are defined modulo $m$, and $\gcd(j_1,j_2,m)=1$. 
We say that $m=m_x$ is the isotropy of $x$, and $\mathbf{j}_x=(m,j_1,j_2)$
are the local invariants for $x$. 

We say that $D\subset X$ is an
isotropy surface of multiplicity $m$ if $D$ is a closed $2$-dimensional suborbifold, and the regular set
$D^\circ\subset D$ is a connected smooth surface with $m_x=m$, for $x\in D^\circ$. 
The \emph{local invariants} for $D$ are those of a point
in $D^\circ$, that is $\mathbf{j}_D=(m,j)$. Locally $D=\{(z_1,0)\}$ and the action is
given by $\varepsilon=e^{2\pi i /m}$, $\varepsilon\cdot (z_1,z_2)=(z_1,\varepsilon^j z_2)$.

Now to describe the action (\ref{eqn:action}) at a point $x$, we 
set $m_1=\gcd(j_1,m)$, $m_2=\gcd(j_2,m)$. Note that $\gcd(m_1,m_2)=1$, so we can
write $m_1m_2 d=m$, for some integer $d$. 
Put $j_1=m_1 e_1$, $j_2=m_2 e_2$. Then we have that \cite[Proposition 2]{MRT}
  $$
  \CC^2/\ZZ_m= ((\CC/\ZZ_{m_2}) \x (\CC/\ZZ_{m_1}) \big)/\ZZ_d\, ,
 $$
where $\CC/\ZZ_{m_2} \times \CC/\ZZ_{m_1}$ is homeomorphic to $\CC^2$ via the map 
$(z_1,z_2)\mapsto (w_1,w_2)=(z_1^{m_2},z_2^{m_1})$. The points of 
$D_1=\{(z_1,0)\}$ and $D_2=\{(0,z_2)\}$ define two surfaces intersecting transversally, and with multiplicities
$m_1,m_2$, respectively, and the action of $\ZZ_d$ on $\CC^2$ is given by
$\varepsilon \cdot(w_1,w_2)=
(e^{2\pi ie_1/d} w_1, e^{2\pi ie_2/d} w_2)$, where $\gcd(e_1,d)=\gcd(e_2,d)=1$. Thus
the point $x$ has as link a lens space ${S}^3/\ZZ_d$, and the images of $D_1$ and $D_2$ are the points
with non-trivial isotropy, with multiplicities $m_1,m_2$, respectively.

We say that $x\in X$ is a singular point if $d>1$ and smooth if $d=1$, and we
denote $d=d_x$. Let $P\subset X$ be the (finite) collection of 
singular points. We say that two surfaces $D_1,D_2\subset X$
\emph{intersect nicely} if at every intersection point $x\in D_1\cap D_2$ there are adapted
coordinates $(z_1,z_2)$ at $x$ such that $D_1=\{(z_1,0)\}$ and $D_2=\{(0,z_2)\}$ in a model
$\CC^2/\ZZ_m$, as above. If the point $x\in X$ is smooth, then $D_1,D_2$ intersect transversally and positively.
In ths situation, the surfaces $D_i$ are said to be \emph{nice}.

A symplectic (cyclic) $4$-orbifold $(X,\omega)$ is a $4$-orbifold $X$ with an orbifold $2$-form 
$\omega\in \Omega^2_{\orb}(X)$ such that 
$d\omega=0$ and $\omega^2>0$. At every point $x\in X$, there are {orbifold Darboux charts} \cite[Proposition 11]{MR}, 
that is an orbifold
chart as above where $\omega$ has the standard form on $\CC^2$ (and hence $\ZZ_m$ acts symplectically). In this
case, the isotropy surfaces $D_i$ are symplectic surfaces (or more accurately, symplectic suborbifolds), and their
intersections are nice, which in this case means that they intersect symplectically orthogonal and positively.  

A K\"ahler (cyclic) $4$-orbifold $(X,J,\omega)$ consists of a symplectic form $\omega$ and a compatible orbifold almost
complex structure $J$, whose Nijenhius tensor $N_J=0$ vanishes. In this case, at every point $x\in X$ there are complex 
charts of the form $\CC^2/\ZZ_m$ as above. The isotropy surfaces $D_i$ are complex curves and the singular points
are (cyclic) complex singularities.

A cyclic orbifold is recovered from the singular points $P\subset X$ and the isotropy surfaces as follows.

\begin{proposition}[{\cite[Propositions 22 and 23]{Mu}}]  \label{prop:smooth->orb}
Let $X$ be an oriented cyclic $4$-orbifold whose isotropy locus is of dimension $0$ (that is, the singular set $P$). Let $D_i$ be 
embedded surfaces intersecting nicely, and 
take coefficients $m_i>1$ such that
$\gcd(m_i,m_j)=1$ if $D_i$, $D_j$ intersect. Then there is a cyclic orbifold structure on $X$, that we denote as $X'$, 
with isotropy surfaces $D_i$ of
multiplicities $m_i$, and singular points $x\in P$ of multiplicity $m_x=d_x\prod_{x\in D_i} m_i$.

Moreover, if $X$ is a K\"ahler (symplectic) cyclic orbifold and the surfaces $D_i$ are complex (symplectic), then
the resulting orbifold $X'$ is a K\"ahler (symplectic) cyclic orbifold. 
\end{proposition}

Once that we have the orbifold $X'$ given in Proposition \ref{prop:smooth->orb}, we need to assign local invariants.
This is not automatic, but the following result is enough for our purposes.

\begin{proposition}[{\cite[Proposition 25]{Mu}}]  \label{prop:local-invar}
Suppose that $X$ is a cyclic $4$-orbifold and the isotropy surfaces $D_i$ which pass through points of $P$
are disjoint. Take integers
$j_i$ with $\gcd(m_i,j_i)=1$ for each $D_i$. Then there exist local invariants for all surfaces $D_i$ and all points $x\in P$.
\end{proposition}

A Seifert bundle over a cyclic $4$-orbifold $X$, endowed with local invariants, is an oriented  
$5$-manifold $M$ equipped with a smooth $S^1$-action such that $X$ is the
space of orbits, and the projection $\pi:M \to X$ satisfies that an orbifold chart $U = \tilde U/\ZZ_m$ of $X$ we have
that 
 $$
 \pi^{-1}(U)= (\tilde U\x S^1)/\ZZ_m\, ,
 $$
where the action of $S^1$ is given by $ \varepsilon\cdot (z_1,z_2,u)=(\varepsilon^{j_2} z_1, \varepsilon^{j_1} z_2,\varepsilon u)$,
and the action in the base is (\ref{eqn:action}).

For a Seifert bundle $\pi:M\to X$, there is a well-defined Chern class \cite[Def.\ 13]{MRT}
 $$
 c_1(M) \in H^2(X,\QQ).
 $$
If we set
 $$
 \ell=\lcm ( m_x \, | \, x\in X), \qquad \mu=\lcm(m_i),
 $$
where $m_i$ are the multiplicitites of $D_i$, then $M/\ZZ_\ell$ is a line bundle over $X$, and $M/\ZZ_\mu$ is a line bundle over $X-P$.
Therefore $c_1(M/\ZZ_\ell)=\ell\, c_1(M)\in H^2(X,\ZZ)$ and $c_1(M/\ZZ_\mu) \in H^2(X-P,\ZZ)$ are
integral classes.

\begin{proposition}[{\cite[Theorem 7.1.3]{BG}}]
Let $(M,\eta , \xi , \Phi , g)$ 
be a quasi-regular K-contact manifold. Then the space 
of leaves $X$ has a natural structure of an almost-K\"ahler cyclic orbifold where the projection $M \to X$ is a Seifert bundle.
Furthermore, if $(M,\eta,\xi,\Phi,g)$ 
is Sasakian, then $X$ is a K\"ahler orbifold. 

Conversely, let $(X,\omega,J,g)$ be an almost-K\"ahler cyclic orbifold with 
$[\omega] \in H^2(X,\QQ)$, and let $\pi: M \to X$ be a Seifert 
bundle with $c_1(M)=[\omega]$. Then $M$ admits a K-contact structure $(M,\eta,\xi,\Phi,g)$
such that $\pi^*(\omega)=d \eta$. 
\end{proposition}

\subsection{Topology of a Seifert bundle}\label{subsec:topology}

Let $X$ be an oriented cyclic $4$-orbifold, $P\subset X$ the set of singular points,
 and $D_i \subset X$ the isotropy surfaces with coefficients $m_i>1$.  
Suppose that there are local invariants ${\mathbf{j}}_{D_i}=(m_i,j_i)$ for each $D_i$ and  
$\mathbf{j}_x$, for every $x\in P$.
Let $0<b_i<m_i$ such that $j_ib_i\equiv 1 \pmod{m_i}$, and  let $B$ be a complex line bundle on $X$. Then
there is a Seifert bundle $\pi:M  \to X$ with the given local invariants and first Chern class 
 \begin{equation}\label{eqn:c1}
 c_1(M)=c_1(B) + \sum_i \frac{b_i}{m_i} [D_i]. 
 \end{equation}

As we aim for simply connected $5$-manifolds $M$, we need to characterize
the first homology group by using the following result.

\begin{proposition}[{\cite[Theorem 36]{Mu}}] \label{thm:16MRT}
Suppose that $\pi:M\to X$ is a quasi-regular Seifert bundle with isotropy surfaces $D_i$ with multiplicities $m_i$,
and singular locus $P\subset X$. 
Let $\mu=\lcm (m_i)$. Then $H_1(M,\ZZ)=0$ if and only if
 \begin{enumerate}
 \item $H_1(X,\ZZ)=0$,
 \item $H^2(X,\ZZ)\to \mathop{\oplus}_i H^2(D_i,\ZZ_{m_i})$ is surjective,
 \item $c_1(M/\ZZ_\mu)\in H^2(X-P,\ZZ)$ is a primitive class.
 \end{enumerate}
 Moreover, $H_2(M,\ZZ)=\ZZ^k\oplus (\mathop{\oplus}_i \ZZ_{m_i}^{2g_i})$, $g_i=g(D_i)$ the genus of $D_i$, $k+1=b_2(X)$.
\end{proposition}

To construct a K-contact manifold from a Seifert bundle, we use the following:

\begin{lemma}[{\cite[Lemma 39]{Mu}}]  \label{lem:20MRT}
Let $(X,\omega)$ be a symplectic cyclic $4$-orbifold with isotropy locus given by surfaces 
$D_i$ and singular locus $P$. Assume given local invariants
$\{{\mathbf{j}}_{D_i}=(m_i,j_i), {\mathbf{j}_x}, x\in P\}$ for $X$.
Let $b_i$ with $j_ib_i\equiv 1 \pmod{m_i}$, $\mu=\lcm(m_i)$. 
Then there is a Seifert bundle $\pi:M\to X$ such that:
\begin{enumerate}
\item It has Chern class $c_1(M)=[\hat\omega]$ for some orbifold symplectic form $\hat\omega$ on $X$.
\item If $\sum \frac{b_i \mu}{m_i} [D_i] \in H^2(X-P,\ZZ)$ is primitive and the second Betti number $b_2(X)\geq 3$, then 
then we can further have that $c_1(M/\ZZ_\mu) \in H^2(X-P,\ZZ)$ is primitive.
\end{enumerate}
\end{lemma}

Finally, in order to control the fundamental group, we introduce the following.

\begin{definition} Let $X$ be an oriented cyclic $4$-orbifold with singular locus $P$ and
isotropy surfaces $D_i$ of multiplicity $m_i$. The orbifold fundamental group $\pi_1^{\orb}(X)$ is defined as 
   $$
   \pi_1^{\orb}(X)=\pi_1\big(X-((\cup D_i)\cup P)\big)/\langle \gamma_i^{m_i}=1\rangle,
   $$
where $\langle\gamma_i^{m_i}=1\rangle$ denotes the relation $\gamma_i^{m_i}=1$ on $\pi_1\big(X-((\cup D_i)\cup P)\big)$,
for any small loop $\gamma_i$ around the surface $D_i$.
\end{definition}

We have the following exact sequence \cite[Theorem 4.3.18]{BG},
  $$
  \cdots\rightarrow \pi_1(S^1)=\ZZ\rightarrow\pi_1(M)\rightarrow\pi_1^{\orb}(X)\rightarrow 1.
  $$
If $H_1(M,\ZZ)=0$ and $\pi_1^{\orb}(X)$ is abelian, then  
$\pi_1(M)$ must be trivial.  This holds since if $H_1(M,\ZZ)=0$, then $\pi_1(M)$ has no abelian quotients. As $\pi_1^{\orb}(X)$ is assumed abelian,
we find that $\pi_1(M)$ is a quotient of $\ZZ$, hence again abelian. This implies that $\pi_1(M)=1$. In such case $M$ is a Smale-Barden manifold.

\subsection{Symplectic constructions}
We review different constructions in symplectic geometry that we will use later. We start with the Gompf symplectic sum \cite{Gompf}. 
Let $S_1$ and $S_2$ be
closed symplectic $4$-manifolds, and $F_1\subset S_1$, $F_2\subset S_2$ symplectic surfaces of
the same genus and with $F_1^2=-F_2^2$. Fix a symplectomorphism $F_1\cong F_2$.
If $\nu_j$ is the normal bundle to $F_j$, then
there is a reversing-orientation bundle isomorphism $\psi: \nu_1\rightarrow \nu_2$. 
Identifying the normal bundles $\nu_j$ with tubular neighbourhoods $\nu(F_j)$ of $F_j$ in $S_i$, one 
has a symplectomorphism  $\varphi: \nu(F_1) -F_1 \rightarrow \nu(F_2) -F_2$ by composing $\psi$ with 
the diffeomorphism $x\mapsto \frac{x}{||x||^2}$ that turns each punctured normal fiber inside out. 
The Gompf symplectic sum is the manifold obtained from 
$(S_1-F_1) \sqcup (S_2-F_2)$ by gluing with $\varphi$ above. It is proved in \cite{Gompf} that 
this surgery yields a symplectic manifold, denoted $S=S_1\#_{F_1=F_2} S_2$.  
The Euler characteristic of the Gompf symplectic sum is given by 
$\chi(S)=\chi(S_1)+\chi(S_2)-2\chi(F)$, where $F=F_1=F_2$.

In \cite[Lemma 24]{MRT}, it is proved that if $D_1\subset S_1$ and $D_2\subset S_2$ are symplectic surfaces intersecting
transversally and positively with $F_1,F_2$, respectively, such that $D_1\cdot F_1=D_2\cdot F_2=d$. Then
$D_1,D_2$ can be glued to a symplectic surface $D=D_1\# D_2 \subset S_1 \#_{F_1=F_2} S_2$
with self-intersection $D^2=D_1^2+D_2^2$ and genus $g(D)=g(D_1)+g(D_2)+d-1$.
This can be done with several surfaces simultaneously.

The following result is very useful to make intersections nice.

\begin{lemma}[{\cite[Lemma 6]{MRT}}] \label{lem:symplectic orthogonal}
Let $(X,\omega)$ be a symplectic $4$-manifold, and suppose that $D, D' \subset X$ are symplectic surfaces 
intersecting transversally and positively. Then we can isotop $D$ (small in the $C^0$-sense, only around the points
of $D\cap D'$) so that the image of $D$ under the isotopy is symplectic, $D$ and $D'$ intersect nicely (symplectically orthogonal).
\end{lemma}

To produce symplectic cyclic $4$-orbifolds, we shall contract chains of symplectic surfaces. 

\begin{definition} \label{def:chain}
Let $X$ be a symplectic $4$-manifold. A \emph{chain of symplectic surfaces} $\cC=C_1\cup \ldots \cup C_l$
consists of $l\geq 1$ symplectic surfaces $C_i$, of genus $g=0$ and self-intersection $C_i^2=-b_i\leq -2$,
such that $C_i\cap C_j=\emptyset$ for $|i-j|>1$ and $C_i\cap C_{i+1}$ is a nice intersection, $i=1,\ldots, l-1$. 
\end{definition}

Note that if we have a chain $\cC=C_1\cup \ldots \cup C_l$ as in Definition \ref{def:chain} where the
intersections $C_i\cap C_{i+1}$ are only transverse and positive, then Lemma \ref{lem:symplectic orthogonal}
allows to perturb the surfaces so that the chain satisfies that the intersections are nice.

\begin{proposition}\label{prop:blow-down}
 Suppose that $\cC=C_1\cup \ldots \cup C_l$ is a chain of symplectic surfaces with $C_i^2=-b_i\leq -2$, $i=1,\ldots, l$.
 Then there is a symplectic cyclic $4$-orbifold $\bar X$ with a singular point $p_0$, and a map $\pi:X\to \bar X$
 such that $\pi^{-1}(p_0)=\cC$, and $\pi:X-\cC \to \bar X-\{p_0\}$ is a symplectomorphism. Moreover
 if we write the continuous fraction $[b_1,\ldots, b_l]=\frac{d}{r}$, $\gcd(d,r)=1$, then the orbifold point $p_0$ is of the form
 $\CC^2/\ZZ_d$, where $\varepsilon\cdot (z_1,z_2)=(\varepsilon z_1,\varepsilon^r z_2)$, $\varepsilon=e^{2\pi i /d}$.
 \end{proposition}

\begin{proof}
Write the continuous fraction
 $$
  \frac{d}{r}=[b_1,\ldots, b_l]=b_1- \frac{1}{b_2-\frac{1}{b_3- \ldots}}
  $$
and consider the action of the cyclic group $\ZZ_d$ on $\CC^2$ given by $(z_1,z_2)\mapsto (\varepsilon z_1,\varepsilon^r z_2)$, where
$\varepsilon=e^{2\pi i /d}$, $0<r<d$ and $\gcd(r,d)=1$.
By \cite[Lemma 15]{CMST}, the complex resolution $\varpi:X'\to \CC^2/\ZZ_d$ 
of $\CC^2/\ZZ_d$ has an exceptional divisor formed by a chain of 
smooth rational curves of self-intersections
$-b_1,-b_2,\ldots,-b_l$. Let $\cC'=C_1'\cup \ldots \cup C_l'$ denote the chain in the K\"ahler manifold $X'$.

In \cite[Theorem 16]{Mu}, it is proven a symplectic neighbourhood theorem for chains of length $l=2$, but it applies 
equally to chains of length $l\geq 2$. It says the following: suppose that $(X,\omega)$, $(X',\omega')$ denote the corresponding
symplectic forms, and suppose that $\la [\omega],  [C_i]\ra =\la [\omega'],[C_i']\ra$, for all $i=1,\ldots, l$ (that is, the
areas of the symplectic surfaces match). Assume also that $C_i^2=C_i'^2$, so that the normal bundles $\nu_{C_i}\cong \nu_{C_i'}$
are isomorphic. Then, there are tubular neighborhoods $\cC\subset U\subset X$ and
$\cC'\subset U'\subset X'$ which are symplectomorphic via $\varphi:U\to U'$, with $\varphi(C_i)=C_i'$, for all $i$.
Then we can take a small ball $B=B_\varepsilon(0)\subset \CC^2/\ZZ_d$ such that $V'=\varpi^{-1}(B)\subset U'$,
and let $V=\varphi^{-1}(V')$. Now we glue $X-\cC$ to $B$ to get a symplectic cyclic $4$-orbifold $\bar X$, with 
a map $\pi: X\to \bar X$ as required.

To arrange the condition on the areas, write $[\omega] =\sum a_i [C_i]$ for $a_i\in \RR$. 
Take $a_{i_0}$ the maximum of the $a_i$. If $a_{i_0}\geq 0$, then 
$\la[\omega],[C_{i_0}]\ra=-a_{i_0} b_{i_0} + a_{i_0-1}+ a_{i_0+1} \leq a_{i_0} (-b_{i_0}+2) \leq 0$, which
icannot occur since the symplectic area $\la [\omega], [C]\ra$ of a symplectic surface $C$ is always positive. 
Hence $a_{i_0}<0$ and therefore all $a_i<0$.
Next we compactify $\CC^2/\ZZ_d$ to $\CP^2/\ZZ_d$, by adding the line at infinity which is away from the 
orbifold point. Consider the resolution $\tilde X$ of $\CP^2/\ZZ_d$, and let $H$ be the hyperplane class. 
As this is projective, it has a K\"ahler class of the form $T=H+ \sum a'_i[C_i']$, with $a'_i<0$. 
The Nakai-Moishezon ampleness criterion
says that $T$ is ample if $T^2>0$ and $T\cdot C>0$ for every effective curve $C$. Hence for every
non-exceptional curve $C$, $H\cdot C\geq \sum (-a'_i) C'_i\cdot C$, and so $C'_i\cdot C\leq m \, H\cdot C$, for some $m>0$.
Now the class $T'=kH+ \sum a_i [C_i'] \in H^2(\tilde X,\RR)$ is a K\"ahler class for $k>0$ large, since 
$T'^2>0$ and $T'\cdot C>0$ for every non-exceptional curve $C$. For $C=C_i'$, $T'\cdot C'_i= \la [\omega], [C_i]\ra >0$.
Then there is a K\"ahler
form $\Omega$ on $\tilde X$ that restricts to a K\"ahler form $\omega'$ on $V'$ such that 
$[\omega']=\sum a_i[C_i']$, and so $\la [\omega'],[C_i']\ra=\la [\omega],[C_i]\ra$.
\end{proof}

Another tool that we need is to transform Lagrangian submanifolds  into symplectic ones.

\begin{lemma}[{\cite[Lemma 27]{MRT}}]\label{lemma:lagr-sympl} Let $(M,\omega)$ be a $4$-dimensional  compact symplectic manifold. 
Assume that  $[F_1],\ldots ,[F_k] \in H_2(M,\ZZ)$ are linearly independent homology classes represented by 
Lagrangian surfaces $F_1,\ldots F_k$ which intersect transversally, not three of them intersect in
a point, and the intersection pattern has no cycles. 
Then there is an arbitrarily $C^\infty$-small perturbation $\omega'$ of the symplectic form $\omega$ such that 
all $F_1,\ldots,F_k$ become symplectic.
\end{lemma}

Note that if the Lagrangian $F_i$ intersects transversally a symplectic surface $S$, after the perturbation we will have
two symplectic surfaces intersecting transversally. With the given conditions, we can 
arrange the orientation of the homology classes $[F_i]$ suitably such
the intersections will be positive, and using Lemma \ref{lem:symplectic orthogonal}, we can make the intersection nice.

\section{Construction of a K-contact Smale-Barden manifold}\label{sec:k-cont}

We are going to start by constructing a simply connected symplectic cyclic $4$-orbifold with $b_2=b^+_2=3$,
and having $3$ symplectic surfaces which are disjoint and span $H_2(X,\QQ)$.

We take the rational elliptic surface $S$ with singular fibers $I_9+3A_1$, that
appears in  \cite[p.\ 568]{Beauville}. To construct $S$, take the pencil of cubic curves in $\CP^2$ with equation
$X^2Y+Y^2Z+Z^2X+tXYZ=0$. We blow-up twice at each of the three points $[1,0,0], [0,1,0],[0,0,1]$, which are
the nodes of the singular curve of the pencil $XYZ=0$. This produces 
a cycle of $9$ curves with another curve intersecting three of them (the image of the smooth curve of
the pencil $X^2Y+Y^2Z+Z^2X=0$), see Figure \ref{fig1}. We blow-up the
three intersections points to get the desired elliptic fibration. There is a cycle of $9$ rational $(-2)$-curves 
$C_1,\ldots, C_9$, and three sections $\s_1,\s_2,\s_3$ with $\s_j$ intersecting $C_{3j+1}$, $j=1,2,3$.
The sections $\s_j$ are rational $(-1)$-curves.
The three nodal curves are given by the values of $t=-3, 3 e^{\pi i/3},-3 e^{\pi i/3}$.

\begin{figure}[H]
\begin{center}
\includegraphics[width=14cm]{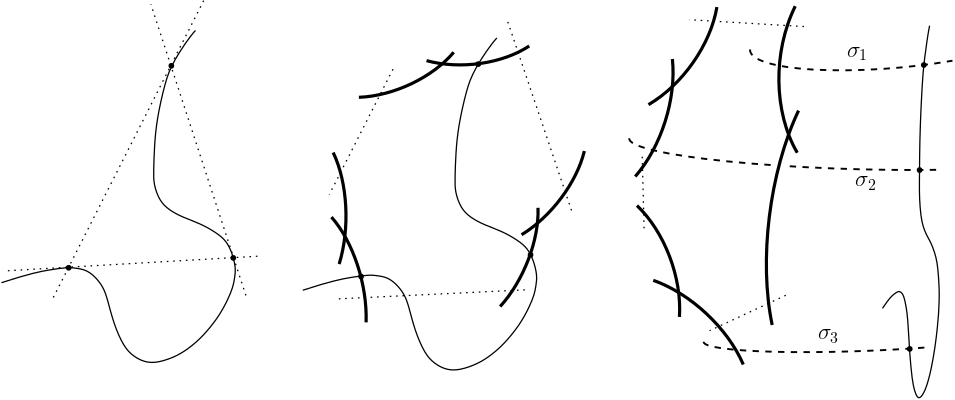}
\caption{\label{fig1} Construction of the rational elliptic surface $S$ with singular fibers $I_9+3A_1$.}
\end{center}
\end{figure}

 Next, let $F$ be a smooth fiber of the elliptic fibration $S\to \CP^1$ obtained from the cubic pencil
 after the blow-ups, and take an isomorphism $H_1(F,\ZZ) \cong \ZZ^2$. The
 monodromy of the fibration, which appears listed in No.\ 63 of \cite[Table 3]{Fukae}, is described by the
 equatlity 
 $CX_{[1,-2]}X_{[2,-1]} A^9 = \mathrm{I}$  (the reverse order is due to the fact that we are understanding the
 matrices as composition of endomorphisms). The notation is
 $X_{[p,q]}=\left(\begin{array}{cc} 1+pq & -p^2 \\ q^2 & 1-pq\end{array}\right)$, and
$A=X_{[1,0]}$, 
 $C=X_{[1,1]}$. 
So the monodromy representation is written as
  $$
  \left(\begin{array}{cc} 2 & -1 \\ 1 & 0\end{array}\right) 
  \left(\begin{array}{cc} -1& -1 \\ 4 & 3\end{array}\right)
  \left(\begin{array}{cc} -1 & -4 \\ 1 & 3\end{array}\right)
  \left(\begin{array}{cc} 1 & -1 \\ 0 & 1\end{array}\right)^9
 = \mathrm{I}.
 $$
The vanishing cycle of $X_{[p,q]}$ is $(p,q)$, with the choice of path to the critical point
taken in \cite{Fukae}. Therefore monodromies corresponding to going around the nodal curves
are $C=X_{[1,1]}$, $X_{[1,-2]}$ and $X_{[2,-1]}$, whence the vanishing cycles are $(1,1)$, $(1,-2)$ and $(2,-1)$.

 We take two copies $S_1,S_2$ of $S$ as above, with two smooth fibers $F_1,F_2$. They have $F_1^2=F_2^2=0$.
 We choose a symplectomorphism $\varphi:F_1 \to F_2$ such that the vanishing cycles match,
 that is, the identity in homology $\varphi_*:H_2(F_1,\ZZ) \to H_2(F_2,\ZZ)$. 
Take the Gompf symplectic sum 
 $$
 X=S_1 \#_{F_1= F_2} S_2 = (S_1-\nu(F_1))\cup_{\nu(F_1)-F_1\cong\nu(F_2)-F_2} (S_2-\nu(F_2))\, .
 $$
As $b_2(S_1)=b_2(S_2)=10$, then $\chi(S_1)=\chi(S_2)=12$. So $\chi(X)=24$ and hence $b_2(X)=22$.

Let $C_1,\ldots, C_9$ be the $I_9$-cycle of $S_1$, with sections $\s_1,\s_2,\s_3$ as before.
Analogously, let $C'_1,\ldots, C'_9$ be the $I_9$-cycle of $S_2$, with sections $\s'_1,\s'_2,\s'_3$ as before.
By using \cite[Lemma 24]{MRT}, we can glue the sections to produce symplectic surfaces $E_1,E_2,E_3$
of square $(-2)$.

\begin{lemma}\label{lem:s-c}
$X$ is simply connected.
\end{lemma}

\begin{proof}
$S_1$ is simply connected, hence $\pi_1(S_1-\nu(F_1))$ is generated by a loop around $F_1$. But this is contracted by using
one of the sections of the elliptic fibration. Hence  $\pi_1(S_1-\nu(F_1))=1$. Also $\pi_1(S_2-\nu(F_2))=1$, so $\pi_1(X)=1$.
\end{proof}

 Fix a fiber $F\subset \bd \nu(F_1)=\bd \nu(F_2)\subset X$. 
 Take the vanishing cycle $a=(1,1)$ in $F$, and the two vanishing thimbles $D_1,D_2$ in $S_1-\nu(F_1), S_2-\nu(F_2)$, respectively.
We glue them to form a Lagrangian $(-2)$-sphere $D$. Next take as dual curve in $F$ the curve
$b=(1,-2)$, intersecting $a$ transversally and positively at three points. This follows since in $H_1(F_0,\ZZ)$, we have
$\la a,b\ra=\la (1,1),(1,-2)\ra =3$, since the intersection form is antisymmetric. Take the torus $T=b\x S^1 \subset
F\x S^1=\bd \nu(F_1)=\bd \nu(F_2)$. This produces a pair of surfaces $D,T$ with 
  \begin{equation}\label{eqn:corr01}
   D^2=-2, \, D\cdot T=3, \,  T^2=0,
   \end{equation}
where $D$ is a Lagrangian sphere and $T$ is a Lagrangian torus.

With the second vanishing cycle $b=(1,-2)$, we do the same thing using another fiber $F'
\subset \bd \nu(F_1)=\bd \nu(F_2)\subset X$. We obtain a pair $D',T'$ of Lagrangian $(-2)$-sphere and Lagrangian
torus, disjoint from $D,T$. To arrange this, first move the $S^1$-factor in $\bd \nu (F_1)= F_1\x S^1$, which
maps to $S^1\subset \CP^1$ around the point giving the fiber $F_1$, in outward direction to make it disjoint. This
produces that $T,T'$ are disjoint. 
Second, note that the vanishing thimbles $D,D'$ map to paths from the point defining the fiber to the point of the nodal 
fiber, and these can be made disjoint; third, to avoid the possible intersections $D\cap T', D'\cap T$ we
use the dual curve to one of the vanishing cycle, which is equal to the other vanishing cycle, and move these loops in
a parallel direction along $F$.
Finally we use Lemma \ref{lemma:lagr-sympl} to change slightly the symplectic form so that $D,D'$ become
symplectic $(-2)$-surfaces, and $T,T'$ become symplectic tori.

Next, we see that there is a chain of $17$ rational $(-2)$-curves
 $$
  \cC= C_8\cup  \ldots\cup  C_2\cup C_1\cup E_1 \cup  C_1'\cup  C_2'\cup  \ldots\cup  C_8'\, ,
  $$  
where $E_1$ is defined before Lemma \ref{lem:s-c} (see Figure \ref{fig2}).
 We contract $D$ and $D'$ to two points $p,p'$ of multiplicity $2$. Using Proposition \ref{prop:blow-down},
 we contract the chain $\cC$ to a
 point $q$ of multiplicity $18$. Note that $[2,\stackrel{(17)}{\ldots}, 2]=\frac{18}{17}$, so the point
 has local model $(z_1,z_2)\mapsto (\varepsilon z_1,\varepsilon^{-1}z_2)$, with $\varepsilon=e^{2\pi i/18}$.
 Denote $\bar X$ the resulting symplectic cyclic orbifold with singular set $P=\{p,p',q\}$. It has
 $b_2(\bar X)=22-2-17=3$, and it is simply connected.

\begin{figure}[H]
\begin{center}
\includegraphics[width=12cm]{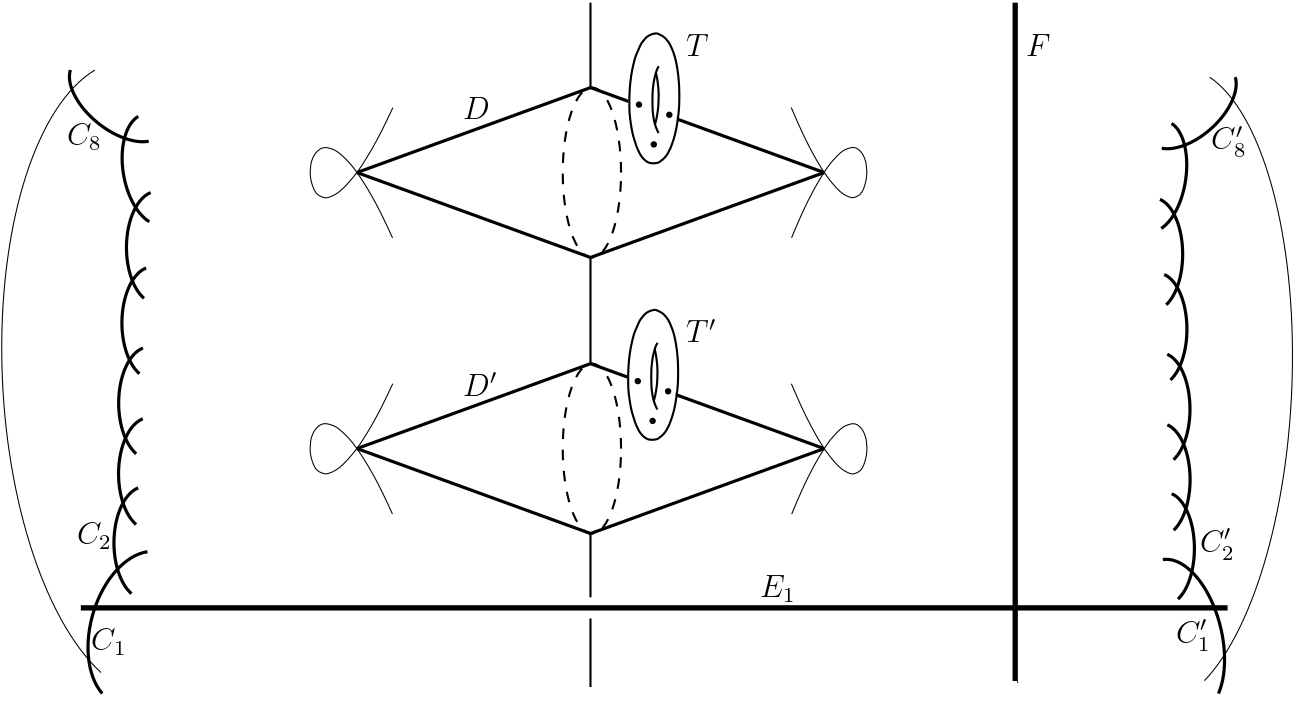}
\caption{\label{fig2} The $4$-manifold $X$ and all the surfaces constructed along the text.}
\end{center}
\end{figure}

  \begin{proposition}\label{prop:theTn}
  There is a collection of smooth symplectic surfaces $T_n$, $n\geq 1$, in a neigbourhood of $T\cup D$,  of genus $g_n=9n^2+1$,
  not intersecting $D$, and such that all the $T_n$ 
  intersect pairwise nicely.
  \end{proposition}
  
  \begin{proof}
Let $K$ be the canonical class of the symplectic form. 
Note that $K\cdot T=0$, $K\cdot D=0$, hence $K\cdot (aD+bT)=0$, for any $a,b\geq 0$.
  
We start constructing a curve $T_1 \equiv 2T+3D$ as follows. By the symplectic neighbourhood
theorem (Proposition \ref{prop:blow-down}), we can assume that we have a holomorphic model
consisting of complex curves $D,T$ in a complex surface. Let $q_1,q_2,q_3$ be the points of
$T\cap D$. We arrange two parallel copies of $T$, say $T',T''$,
which intersect transversely $D$ at six points $q_1',q_2',q_3', q_1'',q_2'',q_3''$, where 
$q_i',q_i'' \in D$ are close to $q_i$, for $i=1,2,3$.
 Take the normal bundle to $D=\CP^1$, which is 
  $\cO_D(-2)$. Take three meromorphic sections $\sigma_1,\sigma_2,\sigma_3$ of $\cO_D(-2)$, where
  $\sigma_i$ has poles at the points $q_i',q_i''$. At each of the six points we do as follows, we do it with
  $q_1'$ for concreteness. Take an adapted chart
  $(z,w)$ around $q_1'$, where $T'=\{z=0\}$, $D=\{w=0\}$. Around $q_1'$ we can assume that $\sigma_1=1/z$. 
  We glue the graph 
  $\{(z,1/z)\}$ with the graph $\{(1/w ,w)\}$ in the normal bundle
  to $T'$ (which is trivial). We use a cut off function to push this graph down to the graph $T'=\{(0,w)\}$,
  as in \cite[Section 3.5]{CMRV}. 
  The result is
  a symplectic surface. This has some self-intersections, which come from the intersections of the sections
  $\sigma_1,\sigma_2,\sigma_3$. These can be resolved symplectically to get a smooth
  symplectic surface of genus $g_1=10$, as in \cite[Section 5.1]{MRT} (basically changing the model $xy=0$ by
  $xy=\epsilon$). Note that 
  $2g_1-2=T_1^2=(2T+3D)^2=12\cdot 3+9\cdot(-2)=18$.
  The surface $T_1$ does not intersect $D$. Note that $T_1\cdot D=(2T+3D)\cdot D=6-6=0$.
  
For given $n\geq 2$, take a collection of symplectic surfaces $\Sigma_1,\ldots,\Sigma_n$ as graphs
in the normal bundle to $T_1$, and all intersecting transversally
and positively. Using \cite[Section 5.1]{MRT}, we can glue symplectically the $\Sigma_i$, $1\leq i \leq n$, at the intersection points
$\Sigma_i\cap \Sigma_j$ to obtain a symplectic surface $T_n\equiv \Sigma_1+\ldots+\Sigma_n \equiv n T_1$. Then 
$T_n^2=18n^2$ and the genus $g_n=9n^2+1$ satisfies $2g_n-2=18n^2$ since $K\cdot T_n=0$. 
 Moreover, if we have different curves $T_n$, all can be taken to intersect transversally, and
  after perturbation as in Lemma \ref{lem:symplectic orthogonal}, the intersections can be arranged to be nice.
 \end{proof}

   \begin{proposition}\label{prop:theA}
 Let $F$ be a fiber of the fibration that intersects the chain $\cC$ transversally at a point of $E_1$. Consider
the configuration of symplectic surfaces $\cC\cup F$. Then there 
is a symplectic surface $A$ of genus $g_A=10$, in a neighbourhood of $\cC \cup F$, not intersecting the chain.
  \end{proposition}
  
  \begin{proof}
Take a cohomology class of the form:
 $$
 A\equiv 2F+ a_0 \sigma+a_1(C_1+C_1')+ a_2 (C_2+C_2')+ \ldots+ a_8(C_8+C_8').
 $$
to arrange that it is disjoint from the curves $E_1$ and $C_i,C_i'$, we need
$0=A\cdot E_1=2-2a_0+2a_1$, $0=A\cdot C_1=a_0-2a_1+a_2$, $0=A\cdot C_i= a_{i-1}-2a_i+a_{i+1}$, $2\leq i\leq 7$, and
 $0=A\cdot C_8=a_7-2a_8$, whose solution is $a_8=1,a_7=2,\ldots, a_1=8, a_0=9$.
Note that 
  $$
  A^2=4a_0+ \sum 4a_{k-1}a_{k} -2a_0^2- \sum 4a_k^2 =18,
  $$
hence $2g-2=K\cdot A+A^2=18$, so $g=10$.

To construct the curve $A$,
consider the push-down map $\pi: X\to \bar X$, which contracts $\cC$ to the singularity $q$. The image 
$A=\pi(A)$ should be a smooth symplectic curve avoiding the singular point. We denote by 
the same letter since it does not pass through the singular point.
We construct $A$ directly in $\bar X$. As noted before, the singularity $q$ is
cyclic of order $18$, and of type $\CC^2/(\varepsilon,\varepsilon^{-1})$, $\varepsilon=e^{2\pi i /18}$. 
As explained in \cite{Reid}, there are $18$ affine charts
covering the $17$ rational curves plus the coordinate axis $L_1=\{y=0\}$, $L_2=\{x=0\}$ (expressed in 
coordinates $(x,y)$). 
Each of these charts is centered at a point of intersection of two consecutive curves in the chain.
They are given by the coordinates:
 \begin{align*}
 &  (\xi_0,\eta_0)= (x^{18}, y/x^{17}), \, (\xi_1,\eta_1)= (x^{17}/y, y^2/x^{16}), \,   (\xi_2,\eta_2)= (x^{16}/y^2, y^3/x^{15}), \ldots \\
 & \ldots,  (\xi_8,\eta_8)= (x^{10}/y^{8}, y^{9}/x^{9})  ,\, (\xi_9,\eta_9)= (x^{9}/y^9, y^{10}/x^{8}), \ldots \\
 & \ldots,  (\xi_{16},\eta_{16})= (x^2/y^{16}, y^{17}/x) , \, (\xi_{17},\eta_{17})= (x/y^{17}, y^{18}) .
 \end{align*}  
 
The axis $L_1$ is $\eta_0=0$ (i.e.\ $y=0$). The $(i+1)$-th curve in the chain is defined by $\xi_i=0$, $i=0,\ldots, 16$. The
second axis $L_2$ is $\xi_{17}=0$ (i.e.\ $x=0$).
The curve $\bar F=\pi(F)$ passing through the mid-point of the $9$-th curve is given by the equation $\eta_8=1$, that is $y^9/x^9=1$. This is
equivalent to $y^9-x^9=0$. The curve $2\bar F$ is thus $(y^9-x^9)^2=0$, that we can perturb to a smooth curve as follows:
 \begin{equation}\label{eqn:perturb}
 (y^9-x^9)^2=\epsilon xy +\zeta,
 \end{equation}
with $0< \zeta \ll \epsilon \ll 1$, in the chart $\CC^2/\ZZ_{18}$. This avoids the singular point (the origin), and it is easily seen
to be smooth. It is $\ZZ_{18}$-equivariant, so it descends to a smooth curve. We have to glue it to two copies of $F$, 
therefore we have to see that the boundary (of the intersection of (\ref{eqn:perturb}) with a ball in the 
affine chart around the singular point) is a collection of two circles. In this way we obtain the curve $A$ sought for.

For proving that the boundary of (\ref{eqn:perturb}) consists of two circles, note that we can see this for $\zeta=0$, since
the extra perturbation will merely move slightly the boundary, and so will not change it topologically.
Note first that $y^9-x^9=0$ is a collection of $9$ lines, interchanged by $\ZZ_{18}$. Actually the image is the quotient
of $y-x=0$ by $(x,y)\mapsto (-x,-y)$, which is the remaining $\ZZ_2$-action. Its boundary is the circle $\{(x,x)| x=e^{it}\}/
(t\sim t+\pi)$. The equation $(y^9-x^9)^2=0$ has as boundary again the same circle, but with multiplicity two.
When we perturb $(y^9-x^9)^2=\epsilon xy$, the curve $(x,y=x)$ gets moved to $(x, y=x+a(x))$, and
we can compute easily a Taylor expansion
 $$
 a(x)= \pm \frac19 \sqrt{\epsilon} \, x^{-7} +\frac{1}{162}\epsilon\, x^{-15} + \ldots
 $$
As we see, there are two solutions depending on the leading term. This means that there are 
two circles in the boundary (the other option would have been a double valued function $a(x)$). As there are
only odd powers of $x$, the 
$\ZZ_2$-action $(x,y)\mapsto (-x,-y)$ goes down to $a\mapsto -a$, via $t\mapsto t+\pi$.
That is, it acts on each circle, and not swapping the circles. So in the quotient, there are two circles
remaining, as claimed.
\end{proof}

Take the push-down map $\pi: X\to \bar X$, which contracts $D,D'$ to singularities $p,p'$ and the
chain $\cC$ to the singularity $q$. Consider the collection of symplectic surfaces $T_n$, $1\leq n\leq N$, and
$T'_m$, $1\leq m\leq N$, and the symplectic surface $A=\pi(A)$. We shall fix a large $N>0$ later on.
None of the surfaces pass through singular points. We arrange all intersections to be nice, so that we
can assign coefficients to all surfaces and make $\bar X$ into a cyclic orbifold $X'$ by using Proposition \ref{prop:smooth->orb}.
We can assign local invariants by using Proposition \ref{prop:local-invar}. 

We take coefficients as follows. 
The genus of $T_n$ and $T_n'$ is $g_n=9n^2+1$, $n\geq 1$. 
For each $1\leq n,m\leq N$, take a prime $p_{nm}$. The collection of chosen primes should be different, $p_{nm}> 3$
and satisfy $p_{nm}>n,m$. We assign multiplicities as follows:
 \begin{align*}
  m_{T_n} &= \prod_{m=1}^N p_{nm}, \\
  m_{T_m'} &= \prod_{n=1}^N p_{nm}^2\, , \\
  m_{A} &= \prod_{n,m=1}^N p_{nm}^3\, .
 \end{align*}
Note that for $T_n,T_s$, $n\neq s$, which are intersecting surfaces, we have that
the primes $p_{nk},p_{sl}$ are different, hence
$\gcd(m_{T_n},m_{T_s})=1$. Analogously, 
$\gcd(m_{T_m'},m_{T_l'})=1$, for $m\neq l$. 
Also note that $\gcd(m_{T_n},m_{A})\neq 1$ and $\gcd(m_{T_m'},m_{A})\neq 1$, for all $n,m\geq 1$.
Also for any $n,m$, $p_{nm}$ divides $m_{T_n}$ and $m_{T_m'}$, hence
$\gcd(m_{T_n},m_{T_m'})\neq 1$. This is in accordance with the fact that the involved surfaces are disjoint.

\begin{theorem} \label{thm:K-contact}
 For any $N\geq 1$, there is a Seifert bundle $\pi:M\to X'$ which is K-contact and satisfies $H_1(M,\ZZ)=0$ and 
 $$
  H_2(M,\ZZ)=\ZZ^2 \oplus \bigoplus_{n,m=1}^N  \left(\ZZ_{p_{nm}}^{18n^2+2}
  \oplus \ZZ_{p_{nm}^2}^{18m^2+2} \oplus \ZZ_{p_{nm}^3}^{20}\right).
 $$
Moreover, $M$ is spin.
 \end{theorem}

\begin{proof}
We need to check the conditions of Proposition \ref{thm:16MRT}. First 
clearly $H_1(\bar X,\ZZ)=0$ because $\pi_1(\bar X)=1$, by Lemma \ref{lem:s-c}. 
Second we have to see the surjectivity of the map
$H^2(\bar X,\ZZ)\to \oplus_i H^2(D_i,\ZZ_{m_i})$. For this, we look at every prime.
Let $p=p_{nm}$ and look at the map
 $$
 \varpi: H^2(\bar X,\ZZ)\to H^2(T_n,\ZZ_p)\oplus H^2(T_m',\ZZ_{p^2})\oplus H^2(A, \ZZ_{p^3}).
 $$
Recall that $T_n\equiv nT_1$, $T_m' \equiv mT'_1$, $A\equiv 2\bar F$ in $H_2(\bar X,\ZZ)$.
The image $\varpi(T_1) =(n T_1\cdot T_1,0,0)=(18n,0,0)$, $\varpi(T_1')=(0,mT'_1\cdot T_1',0)=(0,18m,0)$,
$\varpi(A)=(0,0,A^2)=(0,0,18)$, noting that $T_1,T_1',A\in H^2(\bar X,\ZZ)=H_2(\bar X-P,\ZZ)$.
Now $n,m$ are coprime with $p$ (since $p>n,m$) and $p\geq 5$ so that
$\gcd(p,18)=1$.

To proceed, we need to choose $c_1(M)\in H^2(\bar X,\QQ)$ so that it is a symplectic class, and also
$c_1(M/\ZZ_\mu)\in H^2(\bar X-P,\ZZ)$ is primitive. This follows from Lemma \ref{lem:20MRT} if
we can assure that 
 $$
  x=\mu \left(\sum \frac{b_n}{m_{T_n}}[T_n] +\sum \frac{b'_m}{m_{T_m'}}[T'_m] + \frac{b}{m_{A}} [A]\right) \in H^2(\bar X-P,\ZZ)
 $$
is primitive, where $b_n,b_n'$, $n\geq 1$ and $b$ are the correponding $b_i$ associated to the local invariants.
Note that $\mu=m_{A}=\prod_{n,m} p_{nm}^3$. If we choose $b=1$, then the coefficient of $[A]$ is $1$. 
Cupping with $[A]\in H_2(\bar X-P,\ZZ)$, we obtain $\la x,[A]\ra=[A]^2=18$. So the only possible divisors of
$x$ are $2$ or $3$. Now we note that $T_n\equiv nT_1$.  
Then the coefficient of $T_1$ in $x$ is
 $$
 \frac{b_1 \mu}{m_{T_1}} + \sum_{n\geq 2} \frac{b_n \mu n}{m_{T_n}}\, .
 $$
As $\mu$ is not divisible by $6$, if we choose $b_1=1$ and $b_n$ divisible by $6$ for $n\geq 2$, then this number 
is coprime with $6$. Then $x$ is not divisible by $2$ or $3$, as required.

By \cite[equation (14)]{Gompf}, the second Stiefel-Whitney class of $M$ is 
 $$
  w_2(M)=\pi^*w_2(\bar X-P) + \sum (m_i-1) \pi^{-1}(D_i)\, .
  $$
 As all $m_i$ are odd, then $w_2(M)=\pi^*w_2(\bar X-P)$. Note that $K\cdot T=0$, $K\cdot T'=0$, $K\cdot A=0$,
 hence $K=0$, and so $w_2(\bar X-P)=0$ hence $w_2(M)=0$. So $M$ is spin.
\end{proof}

Finally, we compute the fundamental group.

\begin{theorem} \label{thm:pi1(M)}
The orbifold fundamental group $\pi_1^{\orb}(X')=1$, and $\pi_1(M)=1$.
\end{theorem}

\begin{proof}
  Recall that $X$ is simply connected by Lemma \ref{lem:s-c}, and that we contract the surfaces $D,D'$ and the chain $\cC$.
  The singular points of the orbifold $\bar X$ are $P=\{p,p',q\}$.  Then the fundamental group of 
   $$
   X^o=X-(D\cup D'\cup \cC)=\bar X-P
   $$
 is generated by loops around the singular points, that is $a,a'$ around $p,p'$, respectively, and
 $b$ around $q$. Note that $a^2=1,a'^2=1, b^{18}=1$.

First fix a smooth fiber $F_0$. 
Recall that the vanishing cycles in $F_0$ are $(1,1)$, $(1,-2)$, $(2,-1)$.
Let $\alpha,\beta \in \pi_1(F_0)$ be the standard generators of the torus $F_0$. 
The third vanishing cycle contracts without touching any of the curves, because the vanishing 
thimbles can be taken to be disjoint.
Hence $\alpha^2\beta^{-1}=1$ in $\pi_1(X^o)$, so $\beta=\alpha^2$
and the group generated by $\alpha,\beta$ is generated by $\alpha$ and it is abelian.

Next, take the surface $A$ which lies in a neighbourhood of $F\cup \cC$, and has genus $10$, and self-intersection
 $A^2=18$. Let $\alpha_1,\beta_1,\ldots, \alpha_{10},\beta_{10}$ be the loops 
 generating $\pi_1(A)$ and let $\gamma$ be a small loop around $A$, that is, a meridian. 
 We order the loops so that $\alpha_1,\beta_1,\alpha_2,\beta_2$
 homotop to $\alpha,\beta$ in the fiber $F$, and the other $\alpha_j,\beta_j$ are close to the singular point, so of
 the form $b^{k}$ for some $k$ (the value $k$ depending on the loop). Then
 $\gamma^{18}=\prod_{j=1}^{10} [\alpha_j,\beta_j]=1$. 
 Adding the relation $\gamma^{m_{A}}$, and recalling that $\gcd(m_{A},18)=1$ (since we
 have chosen all primes $p>3$), we get $\gamma=1$ in
 $\pi_1^{\orb}(X')$.

Now the fiber $F_0$ intersects the chain $\cC$ in the central curve $E_1$. Note
that the loops around the curves $C_8, \ldots,  C_1,  E_1 ,  C_1' , \ldots ,  C_8'$ are
given as, in this order, $b, \ldots, b^8, b^9, b^{10},\ldots, b^{17}$. Then the loop
around $E_1$ is $b^9$, which produces the relation $b^9=[\alpha,\beta]=1$.
Now we use the fact that there are two extra sections $E_2$, $E_3$ of the fibration.
These avoid $T,D,T',D'$, $E_2$ intersects $C_4,C_4'$, and $E_3$ intersects $C_7,C_7'$.
They intersect $A$ in two points. The loop around $A$ is trivial $\gamma=1\in \pi_1^{\orb}(X')$.
So we get relations $b^5=b^{13}$ and $b^2=b^{16}$. 
So $b^8=1$, that together with $b^9=1$ imply that $b=1$. 
 
Finally take the surfaces $T_n$, $n \geq 1$, lying  in a neighbourhood of 
$T\cup D$. Let $c_j$ be a small loop around $T_j$, $j\geq 0$, and recall
 that $a$ is a small loop around $D$, and $a^2=1$. All curves $T_j$ intersect transversally, so $[c_j,c_k]=1$, for all $j,k$. 
 Let $c_0$ be a small loop around $T$, then $c_0=a^3=a$, since $T\cdot D=3$.
 Move $T$ slightly off to get a relation 
  $c_1^9c_2^{18}\ldots c_N^{9N}a=[\alpha\beta^{-2} ,\gamma]=1$
 (the last loops are the generators of $\pi_1(T)$), using that $T\cdot T_k=9k$.
The group generated by $a, c_1,\ldots, c_N$ is abelian. 
We write the relations additively,
  \begin{equation}\label{eqn:relat}
  9 (c_1+2c_2+\ldots +N c_N)+a=0 .
   \end{equation}

Put for brevity, $m_j=m_{T_j}$. In $\pi_1^{\orb}( X')$ we have the 
extra relations $m_j c_j=0$. 
 Multiply (\ref{eqn:relat}) by $M_k=\prod_{j\neq k} m_j $ to get $9M_k k c_k=0$. 
 Now $m_kc_k=0$, and $\gcd(m_k,M_k)=1$ since $\gcd(m_j,m_k)=1$ for $j\neq k$,
 and also $\gcd(k,m_k)=1$ because we chose $\gcd(p_{km},k)=1$, and $m_k=\prod_m p_{km}$, and
 $\gcd(m_k,3)=1$ as all primes are $p>3$.
 All together give $c_k=0$, for $k\geq 1$. Then (\ref{eqn:relat}) gives $a=0$ 
 as well in $\pi_1^{\orb}( X')$.

Once that $\pi_1^{\orb}(X')=1$, we get $\pi_1(M)=0$ by the argument at the end of Section \ref{subsec:topology}
using that $H_1(M,\ZZ)=0$ from Theorem \ref{thm:K-contact}.
\end{proof}

Therefore the K-contact manifold from Theorem \ref{thm:K-contact} is a Smale-Barden manifold.

\section{Bounding the number of singular points} \label{sec:4}

Our last step is to prove that the manifold $M$ from Theorem \ref{thm:K-contact} 
does not admit a Sasakian structure. 
Suppose that $M$ admits a Sasakian structure. Then there is Seifert bundle
 $$
  \pi: M\to Y,
  $$
where $Y$ satisfies the following conditions that we state explicitly:

\begin{conditions}\label{condi}
$Y$ is a K\"ahler cyclic orbifold with $b_2=3$, $b_1=0$. Associated to each prime $p_{nm}$, $1\leq n,m\leq N$, there
is a collection of three complex curves
 \begin{equation}\label{eqn:Dnm}
  D_1^{nm}, D_2^{nm}, D_3^{nm},
  \end{equation}
which have genus $g( D_1^{nm})=9n^2+1$, $g(D_2^{nm})=9m^2+1$, $g(D_3^{nm})=10$. For each $(n,m)$,
the three curves (\ref{eqn:Dnm}) are disjoint, and span $H_2(Y,\QQ)$. 
Moreover, these curves are all nice, and intersect pairwise nicely.
\end{conditions}

This follows from the homology of $M$ appearing in Theorem \ref{thm:K-contact} and the
relation to the homology of the base of a Seifert bundle given in Proposition \ref{thm:16MRT}.
The curves (\ref{eqn:Dnm}) are the components of the isotropy locus. 
Conditions \ref{condi} imply that at most two different curves can go through a point 
of the singular set $P\subset X$. Some of the curves could be equal 
 for different values of $(n,m)$, e.g.\  $D_1^{nk}=D_1^{nl}$, $k\neq l$; or
 $D_2^{km}=D_2^{lm}$, $k\neq l$; or  $D_1^{nk}=D_2^{ln}$; or $D_3^{nm}=D_3^{kl}$.
Clearly, this can only happen if the genera of the involved curves are the same.

To prove that $M$ cannot admit a Sasakian structure, we are going to get a contradiction if
we assume the existence of a 
K\"ahler orbifold  $Y$ satisfying Conditions \ref{condi}, for some $N\gg 0$ large enough. After preparatory
work in Sections \ref{sec:4}, \ref{sec:Many} and \ref{sec:6}, this will be proved in Theorem \ref{thm:main-Sasakian} in
Section \ref{sec:7}.

To start with, let $Y$ be a K\"ahler cyclic orbifold satisfying Conditions \ref{condi}. We do not assume
$\pi_1^{\orb}(Y)=0$. 
Our first task is to obtain a universal bound on
the number of singular points $\#P$.

Let $\pi:\tilde Y\to Y$ be the minimal resolution of singularities. For every cyclic singularity $p\in Y$,
$\pi^{-1}(p)=E_p=C_1\cup \ldots \cup C_l$ is a chain of rational curves with self-intersection $C_j^2=-b_j \leq -2$.
For any curve $A$ in $Y$, we denote the proper transform as $\tilde A$.
Let $A$ be a nice curve through $p$ (for the sake of simplicity, assume that there is only one singular point). 
Then $\tilde A$ intersects transversely just one of the extremal curves of the chain
$C_1,C_l$. For concreteness, say it is $C_1$. We have that (see \cite[p.\ 80]{BPV})
 $$
  \pi^*A =\tilde A+ \sum r_{i} C_i\, ,
  $$
where $\frac{r_{k+1}}{r_{k}}= [b_{k+1},\ldots, b_l]^{-1}$, for $0\leq k\leq l-1$, where $r_0=1$. Note that 
$\frac{r_{k+1}}{r_k}<1$, hence $0<r_l<r_{l-1}<\ldots < r_1<1$.
Next, let $\tilde K$ be the canonical divisor of $\tilde Y$, and $K$ the canonical divisor of $Y$. In this case,
$\tilde K$ is not the proper transform of $K$. We have a formula (again assuming only one singular point)
 \begin{equation}\label{eqn:tildeK}
  \tilde K=\pi^* K - \sum \lambda_i C_i\, ,
 \end{equation}
where $\lambda_i\geq 0$. If there are more singular points, then we have to add the contribution over
each $p\in P$.
  
\begin{lemma}\label{lem:lem}
 Let $A,B$ be two effective divisors in $Y$, and let $\tilde A,\tilde B$ be the
 proper transforms. Then $A\cdot B\geq \tilde A\cdot \tilde B$.
\end{lemma}

\begin{proof}
It is enough to prove it for $A,B$ two irreducible curves in $Y$.
Then
 \begin{align*}
 & A\cdot B=\pi^*A \cdot \pi^* B=\pi^*A \cdot \tilde B=(\tilde A+E)\cdot \tilde B
 \geq \tilde A\cdot \tilde B.
  \end{align*}
 The second equality follows since $\pi^*A\cdot C_i=A\cdot \pi_*C_i=0$, for any exceptional divisor $C_i$.
 The third, because $\pi^*A=\tilde A+E$, where $E=\sum r_iC_i$, where $r_i\in \QQ$, $r_i\geq 0$,
 and $C_i$ are exceptional divisors. The last equality is due to $C_i\cdot \tilde B\geq 0$, being
 two distinct effective curves.
\end{proof}

Now let $D_1,D_2,D_3$ be three nice curves, which are disjoint, and span $H_2(Y,\QQ)$.  We have the following:

\begin{lemma}\label{lem:Sigma}
 The $\QQ$-divisor $K+D_1+D_2+D_3$ is effective. Also $K+D_i$ is effective, $i=1,2,3$.
 \end{lemma}

\begin{proof}
We have the exact sequence
 $$
  0 \to \cO(\tilde K) \to \cO(\tilde K+ \tilde D_1+ \tilde D_2+\tilde D_3) \to \oplus \, \cO_{\tilde D_i}(K_{\tilde D_i}) \to 0,
  $$
because the $\tilde D_i$ are disjoint, and using the adjunction formula. 
As $g(\tilde D_i)=g_i\geq 1$, $H^0( \cO_{\tilde D_i}(K_{\tilde D_i}))=\CC^{g_i}\neq 0$, by Riemann-Roch. 
As $b_1(\tilde Y)=0$, so $H^1(\tilde K)=0$. Also $H^0(\tilde K)=0$, since $h^{0,2}(\tilde Y)=0$, because
$H^2(\tilde Y,\CC)$ is spanned by complex curves, so $h^{1,1}(\tilde Y)=b_2(\tilde Y)$.
Therefore $h^0(\cO(\tilde K+ \tilde D_1+ \tilde D_2+\tilde D_3))=g_1+g_2+g_3>0$, and hence there is some effective divisor
$\Sigma'\equiv \tilde K+ \tilde D_1+ \tilde D_2+\tilde D_3$. Pushing down, $\Sigma=\pi( \Sigma' )\equiv K+D_1+D_2+D_3$ in $Y$.

The last assertion is proved in the same way.
\end{proof}

Put $B=D_1+D_2+D_3$. We need to check that $K+B$ is log canonical, whose definition appears in \cite[Definition 1.16]{Ast}.
This is checked at each singular point $p\in P$. Suppose that $p\in D_1=D$ (the case that $p$ is not in any divisor $D_i$ is 
similar). Assume for simplicity that there are no more singular points, and write
 $$
 \tilde K =\pi^*(K+D) -\tilde D +\sum_{i=1}^{l} a_i C_i\, ,
 $$
where $\tilde D$ is the proper transform of $D$, 
$C_i$ are the exceptional divisors, ordered so that $\tilde D\cdot C_1=1$. 
We have to check that $a_i\geq -1$. Letting $a_0=-1$ the coefficient of $\tilde D$,
and setting $a_{l+1}=0$, we have the equalities $\tilde K\cdot C_i+C_i^2=-2$, which give 
$-a_ib_i + a_{i-1}+a_{i+1} - b_i=-2$. This is rewriten as
  \begin{equation}\label{eqn:ai0}
 (a_{i-1}-a_i) + (a_{i+1}-a_i) = (a_i+1) (b_i-2).
 \end{equation}
Let $i_0$ be such that $a_{i_0}$ is minimum. Then the left hand side of (\ref{eqn:ai0}) for $i=i_0$ is $\geq 0$. Therefore 
if $b_{i_0}>2$, then $a_{i_0}+1\geq 0$, and so $a_i\geq a_{i_0}\geq -1$, for all $i$.
If $b_{i_0}=2$ then $a_{i-1}=a_i=a_{i+1}$ and we proceed recursively. 

Recall the definition of the orbifold Euler-Poincar\'e characteristic of an orbifold with isolated singularities,
$$
 e_{\orb}(Y)= e(Y) - \sum_{p\in P} \left(1- \frac{1}{d_p}\right),
$$
where $d_p$ denotes the multiplicity of the singular point $p\in P$.

\begin{theorem} \label{prop:mas-mas}
Now let $D_1,D_2,D_3$ be three disjoint nice curves that span $H_2(Y,\QQ)$.  
Then $e_{\orb}(Y-(D_1\cup D_2\cup D_3)) \geq 0$. 
\end{theorem}

\begin{proof}
Let $B=D_1+D_2+D_3$.
We already know that $(Y,B)$ is log canonical and effective.
If we have that $K+B$  is nef, then \cite[Theorem 10.14]{Ast} implies that
 $$
  3c_2 (\Omega_Y^1(\log B)) \geq c_1(\Omega_Y^1(\log B))^2 = (K+B)^2 \geq 0,
  $$
where the last inequality is due to the fact that $K+B$ is effective and nef.
By \cite[Theorem 10.8]{Ast}, we have that $c_2 (\Omega_Y^1(\log B)) =e_{\orb}(Y-B)$ and the result follows.

It remains to see that $K+B$ is nef. Let $A\subset Y$ be an irreducible curve, and let us check that
$(K+B)\cdot A \geq 0$.
First assume that $A=D_i$. By Lemma \ref{lem:lem}, $K+D_i$ is effective. Moreover $H^0(\tilde K+\tilde D_i) \to 
H^0(\cO_{\tilde D_i}(K_{\tilde D_i}))$ is bijective, and as $K_{\tilde D_i}$ is base-point free,
$\tilde K+\tilde D_i$ can be represented by a divisor not containing $\tilde D_i$. Hence there
is $C\equiv K+D_i$ not containing $D_i$, and thus
 $(K+B)\cdot A= (K+D_i)\cdot D_i =C\cdot D_i \geq 0$.

So we can suppose now that $A\neq D_i$, $i=1,2,3$. 
Let $\Sigma \equiv K+B$ be an effective $\QQ$-divisor. Write $\Sigma=r A+ T$, where $T$ does
not contain $A$. If $A^2\geq 0$ then $(K+B)\cdot A =\Sigma\cdot A\geq 0$, and we are done.
So we can assume that $A^2<0$. By Lemma \ref{lem:lem}, we also have $\tilde A^2<0$.

Next suppose that $\tilde K\cdot \tilde A\geq 0$. Then as $K=\pi_*\tilde K$ (although
$\tilde K$ is not the strict transform of $K$), 
 $$
 K \cdot A =\pi_*\tilde K\cdot A=\tilde K\cdot \pi^*A=\tilde K\cdot (\tilde A+E)\geq \tilde K\cdot \tilde A\geq 0,
 $$
where $E$ is an effective $\QQ$-divisor consisting of exceptional curves $E=\sum r_iC_i$, $r_i\geq 0$.
For any $C_i$, it is $\tilde K\cdot C_i\geq0$, since they are rational curves with $C_i^2\leq -2$.
As $B\cdot A\geq 0$ then $(K+B)\cdot A\geq 0$, as required.

So we are left with the case of an irreducible curve $\tilde A$ with $\tilde A^2<0$, $\tilde K\cdot \tilde A<0$. 
As $p_a(\tilde A)=\tilde A^2 + \tilde K\cdot \tilde A <0$, then $\tilde A$ must be a smooth
rational curve, with $\tilde A^2=-1$ and $\tilde K\cdot \tilde A=-1$, that is an exceptional divisor for a minimal 
model of $\tilde Y$.

If $A\cdot D_1=A\cdot D_2=0$, then $A\equiv \lambda D_3$, for some $\lambda >0$, which
is impossible since $A\cdot D_3\geq 0$ and $D_3^2<0$.
If $A\cdot D_2=A\cdot D_3=0$, then $A\equiv \lambda D_1$, for some $\lambda >0$, and hence
$A^2>0$, contrary to our current assumption.
Finally, 
if $A\cdot D_1=0$, then $A\equiv \lambda_2 D_2+\lambda_3 D_3$, for some $\lambda_i\in \QQ$.
Since $A\cdot D_i\geq 0$, then $\lambda_i\leq 0$, $i=2,3$. Hence $A\leq 0$, which is a 
contradiction as $A$ is effective.

So we can assume  $A\cdot D_1>0$ and $A\cdot D_2>0$ (after swapping $D_2,D_3$ if necessary). 
Now $(\tilde K+\tilde D_1+\tilde D_2+\tilde D_3 )\cdot \tilde A=-1 +\sum \tilde D_i \cdot \tilde A$.
 If $\tilde D_i\cdot \tilde A\geq 1$ for some $i$, then 
 $(\tilde K+\tilde D_1+\tilde D_2+\tilde D_3 )\cdot \tilde A\geq 0$. Recall that there is an
 effective $\Sigma'\equiv \tilde K+\tilde D_1+\tilde D_2+\tilde D_3$, with $\pi(\Sigma')=\Sigma \equiv K+B$.
 So 
  $$
  (K+B)\cdot A=\Sigma'\cdot \pi^*A=\Sigma'\cdot (\tilde A+E) \geq 0,
  $$ 
where we have $E=\sum r_i C_i$, $r_i\geq 0$, and 
 $\Sigma'\cdot C_i\geq 0$ because $\tilde D_j\cdot C_i\geq 0$ and $\tilde K\cdot C_i\geq 0$.
 
Hence we can further assume that $\tilde D_i\cdot \tilde A=0$, for all $i$. 
 As $A\cdot D_1>0$, $A\cdot D_2>0$,
 this means that $\tilde A$ intersects a chain of exceptional divisors $E_p$,
 for a singularity $p\in A\cap D_i$, for both cases $i=1,2$.
 By Lemma \ref{lem:lem}, we have
 $$
  (K+D_1+D_2+D_3+A) \cdot A \geq 
 (\tilde K+\tilde D_1+\tilde D_2+\tilde D_3+\tilde A)\cdot \tilde A + \sum_{p\in P} \ell_p
 =-2 +\sum_{p\in P} \ell_p,
  $$
where $\ell_p$ denotes a local contribution of the intersection at $p$. There is contribution to 
$\ell_p$ only if $p\in A\cap D_i$. This happens at least for two singular points, hence
it is enough to see that $\ell_p\geq 1$ if
$p\in A\cap D_i$. 
Once we have checked that, we have that $(K+D_1+D_2+D_3+A)\cdot A\geq 0$. 
As we are assuming $A^2<0$, we have
$(K+D_1+D_2+D_3)\cdot A\geq 0$, i.e.\ $(K+B)\cdot A\geq 0$, completing the proof.

Let us finally see that $\ell_p\geq 1$.
The proper transform $\tilde A$ intersects the chain $E_p=C_1\cup \ldots\cup C_l$,
but not $\tilde D_i$.  
Let $\alpha_j=\tilde A\cdot C_j \in \ZZ_{\ge 0}$.
Take (a germ of) a curve $A'_j$ that intersects transversally $C_j$ and no other $C_k$. Then
$\tilde A\equiv \sum \alpha_j A'_j$ in a neighbourhood of $E_p$, and so $A\equiv \sum \alpha_j A_j$ where $A_j=\pi(A_j')$.
Then the contribution at $p$ is
 \begin{align*}
 \ell_p =& \left(  (K+D_1+D_2+D_3+A)\cdot A \right)_p \\  = &
 \left(  (K+D_1+D_2+D_3+\sum \alpha_j A_j)\cdot \sum \alpha_k A_k \right)_p \\
 =& \sum \alpha_j \left(  (K+D_1+D_2+D_3+ A_j)\cdot A_j \right)_p  \\ & + \sum_{j\neq k}\alpha_j\alpha_k (A_j\cdot A_k)_p
 +\sum (\alpha_j^2-\alpha_j)(A_j^2)_p\, .
  \end{align*}
The local intersection number is defined in \cite{Cogo}. As $(A_j\cdot A_k)_p\geq 0$, we see that it is 
enough to prove the result for $A=A_j$. We assume this henceforth.

To compute $\left(  (K + D_i + A)\cdot A \right)_p$, note that $A$ only intersects $C_j$. We contract
$C_1\cup \ldots \cup C_{j-1}$ and $C_{j+1}\cup \ldots\cup C_l$ and get an orbifold $\bar Y$, such that
there are contractions $\tilde Y\to \bar Y\to Y$. The map $\varpi:\bar Y\to Y$ has an exceptional
divisor $\bar E$ with two orbifold points $p_1,p_2$ of multiplicities
$d_1,d_2$ respectively (it is $d_1=1$ if $j=1$, and $d_2=1$ if $j=l$). The proper transform of $D_i$ is
$\bar D_i$ with $p_1=\bar D_i\cap \bar E$, which is a nice intersection. The proper transform
of $A$, denoted $A$ again, intersects $\bar E$ transversally at a smooth point.

We have the following intersection numbers (see \cite{Cogo}). Let 
$\frac{a_1}{d_1}=[b_{j-1},\ldots,b_1]^{-1}$, 
$\frac{a_2}{d_2}=[b_{j+1},\ldots, b_l]^{-1}$ be the
continuous fractions  associated to the singularities (with $a_j=0$ if $d_j=1$), and 
let 
$\frac{a'_1}{d_1}=[b_{1},\ldots,b_{j-1}]^{-1}$, 
$\frac{a'_2}{d_2}=[b_{l},\ldots, b_{j+1}]^{-1}$ 
be the dual ones. Then, writing $b=b_j$,
 \begin{align*}
  (\bar D_i\cdot \bar E)_{p_1} &= \frac{1}{d_1} \\
  \bar E^2 &= -b+ \frac{a_1}{d_1} +\frac{a_2}{d_2} \\
  (\bar D_i^2)_{p_1} &= \frac{a_1'}{d_1} 
\end{align*}
Using the adjunction equality for a nice curve,
$$
K_{\bar Y}\cdot C+ C^2=-e_{\orb}(C)=2g(C)-2+\sum_{p\in C} \left(1-\frac{1}{d_p}\right),
$$
the corresponding local contribution for $C=\bar E, \bar D_i$, gives
 \begin{align*}
 K_{\bar Y}\cdot \bar E &= b-\frac{a_1+1}{d_1}-\frac{a_2+1}{d_2} \\
 (K_{\bar Y}\cdot \bar D_i)_p &= 1-\frac{a_1'+1}{d_1}
\end{align*}

Recall that we aim to compute $\ell_p=\left((K+D_i+A)\cdot A\right)_p=\left((\bar K+\bar D_i+A)\cdot \varpi^*A\right)_p$,
where the right hand side accounts for the contribution to the intersection along the exceptional divisor.
We write $\varpi^*A=A+ x \bar E$, and compute $x\in \QQ$ knowing that $\varpi^*A\cdot \bar E=0$
and $A\cdot \bar E=1$. Then 
 $$
 x=-\frac{1}{\bar E^2}= \frac{d_1d_2}{bd_1d_2-a_1d_2-a_2d_1}\, ,
 $$
and hence
 \begin{align*}
 \ell_p &= \left((K+D_i+A)\cdot A\right)_p = \left((\bar K+\bar D_i+A)\cdot \varpi^*A \right)_p \\
 &= \left( (\bar K+\bar D_i+A)\cdot (A+ x\bar E)\right)_p \\
 &= \left( b-\frac{a_1+1}{d_1}-\frac{a_2+1}{d_2} +\frac{1}{d_1}+1 \right)\frac{d_1d_2}{bd_1d_2-a_1d_2-a_2d_1} \\
  &=1+ \frac{d_1(d_2-1)}{bd_1d_2-a_1d_2-a_2d_1} \geq 1,
  \end{align*}
as required. 
\end{proof}

By Theorem \ref{prop:mas-mas}, the orbifold Euler-Poincar\'e characteristic  is
 $$
 e_{\orb}(Y-(D_1\cup D_2\cup D_3)) = 5 - \sum (2-2g_i) 
 - \sum_{p\in Y \atop p\notin D_1\cup D_2\cup D_2} \left(1-\frac{1}{d_p}\right) \geq 0,
 $$
 where $g_1,g_2,g_3$ are the genus of $D_1,D_2,D_3$, respectively.
As $d_p \geq 2$, we deduce 
 \begin{equation}\label{eqn:bounddd}
  \# \{p\in P, p\notin D_1\cup D_2\cup D_2\} \leq 2(2g_1+2g_2+2g_3-1).
 \end{equation}

Now let us have three collections of curves $(D_1,D_2,D_3)$, $(D'_1,D'_2,D'_3)$, $(D''_1,D''_2,D''_3)$
in the same situation, and suppose that all curves are distinct. Let $A=\{p\in P, p\in D_1\cup D_2\cup D_3\} ,
A'=\{p\in P, p\in D_1'\cup D_2'\cup D_3'\} ,
A''=\{p\in P, p\in D_1''\cup D_2''\cup D_3''\}$.
If $p\in Y$, then it can be at most in two curves (since they intersect nicely). Therefere either $p\notin
A, p\notin A'$ or $p\notin A''$. 
So $P\subset (P-A)\cup (P-A')\cup (P-A'')$.
By the equality (\ref{eqn:bounddd}) above,
 \begin{equation}\label{eqn:T0}
  \# P \leq 2(2g_1+2g_2+2g_3-1)+2(2g'_1+2g'_2+2g'_3-1)+2(2g''_1+2g''_2+2g''_3-1),
 \end{equation}
where $g_i,g_i',g_i''$ denote the genus of the respective curves. 

\begin{corollary}\label{cor:Pbounded}
 Suppose that we have five bases with genera $\{g_1,g_2,10\}$, $\{g_1',g_2',10\}$, $\{g_1'',g_2'',10\}$,
 $\{g_1''',g_2''',10\}$, $\{g_1'''',g_2'''',10\}$, 
 and all $g_1,g_2$, $g_1',g_2'$, $g_1'',g_2''$, $g_1''',g_2'''$, $g_1'''',g_2''''$ and 
 $10$ are distinct numbers. Then there is some $\cota_0$ (independent
 of $Y$) such that $\#P\leq \cota_0$.
\end{corollary}

\begin{proof}
For checking this, we use Definition \ref{lem:proj-equiv} from upcoming Section \ref{sec:Many} (the results that 
we use for this proof are independent of Section \ref{sec:Many}). 
If among the five curves of genus $10$, say $D_3,D_3',D_3'',D_3''',D_3''''$, there
are only two distinct curves, then three of them coincide. Suppose that $D_3=D_3'=D_3''$. 
Then Lemma \ref{lem:1-b} (below) implies that two of the bases 
 (say $\varepsilon,\varepsilon'$) are
 proj-equivalent, and hence $D_2'=\lambda_2D_2$, with $\lambda_2>0$. As $D_2\neq D_2'$ because
 they have different genus, we get $D_2\cdot D_2'\geq 0$. But then $\lambda_2 D_2^2\geq 0$, which
 is a contradiction, as $D_2^2<0$.
 
Therefore there are three of the bases with all curves distinct, and we take $\cota_0$ as the 
right hand side of the formula (\ref{eqn:T0}).
\end{proof}

The assumption of Corollary \ref{cor:Pbounded} is achieved as soon as we take $N\geq 11$
for (\ref{eqn:Dnm}).

\section{Many collections of orthogonal bases of curves} \label{sec:Many}

Let $Y$ be a  K\"ahler cyclic orbifold with $b_1=0$ and $b_2=3$. Let $P$ be the collection of singular points.
Suppose that the ramification locus consists of a collection of nice curves $D_i^{(k)}$, $i=1,2,3$, $1\leq k\leq K$,
such that
 \begin{equation}\label{eqn:varepsilonk}
  \varepsilon^{(k)}=( D_1^{(k)},D_2^{(k)},D_3^{(k)})
  \end{equation}
are orthogonal bases for $H_2(Y,\QQ)$, formed by curves which are disjoint. 
As $Y$ is a K\"ahler orbifold, $h^{1,1}(Y)=b_2(Y)=3$, because the homology is spanned by complex curves.
The intersection form of $H^2(Y,\RR)$ is of signature $(1,2)$. So we can order the curves so that
$(D_1^{(k)})^2=m_1^{(k)}>0$, $(D_2^{(k)})^2=-m_2^{(k)}<0$, $(D_3^{(k)})^2=-m_3^{(k)}<0$.
The genera are $g_1^{(k)}=g(D_1^{(k)}),g_2^{(k)}=g(D_2^{(k)}),g_3^{(k)}=g(D_3^{(k)}) \geq 1$.

For $k\neq l$, it may happen that $D_i^{(k)}=D_j^{(l)}$ in which case $g_i^{(k)}=g_j^{(l)}$,
and also either $i=j=1$ or $i,j\in \{2,3\}$ (since the self-intersection coincides).
On the other hand, if the curves are distinct then it must be $D_i^{(k)} \cdot D_j^{(l)} \geq 0$.

\begin{definition}\label{lem:proj-equiv}
Let $\varepsilon= \big(D_1,D_2, D_3\big)$, $\varepsilon'= \big(D'_1,D_2', D'_3\big)$ be two bases
from the above list. We write $[\varepsilon]=[\varepsilon']$ 
if the elements are proportional, that is up to reordering, $D_i'=\lambda_i D_i$ with $\lambda_i>0$. 
We say that the bases are \emph{proj-equivalent}.
\end{definition}

Note that if $[\varepsilon]=[\varepsilon']$ then, 
by the discussion above, we have that $D_2=D_2'$ and $D_3=D_3'$. 

\medskip

Let $K$ be the orbifold canonical class of $Y$. Let  $\varepsilon= ( D_1,D_2, D_3 )$ be one of the basis
provided above. Then we write $K=\sum a_i D_i$. We have the orbifold adjunction equality
 $$
 K\cdot D+D^2=-e_{\orb}(D),
 $$ 
for a smooth orbifold (nice) curve $D\subset Y$. 
As $b^+_2=1$, we have that $D_{1}^2=m_1>0$, $D_i^2=-m_i<0$ for $i= 2,3$, where $m_i\in\QQ$. Let $g_i$ 
be the genus of $D_i$. Let 
 $$
 \chi_i=-e_{\orb}(D_i)=2g_i-2+\sum_{p\in D_i} \left(1-\frac{1}{d_p}\right),
 $$
where $d_p$ is the
order of the singular point $p\in D_i$. Then $\chi_i\geq 2g_i-2$. Using the adjunction formula, then 
$a_1=(\chi_1-m_1)/m_1$, $a_i=-(\chi_i+m_i)/m_i$ for $i=2,3$, so
 \begin{equation}\label{eqn:K}
  K=\frac{\chi_1-m_1}{m_1} D_1 - \frac{\chi_2+m_2}{m_2} D_2 - \frac{\chi_3+m_3}{m_3} D_3\, .
  \end{equation}
Note that $\chi_i+m_i>0$ for  $i= 2,3$.

By Lemma \ref{lem:Sigma}, $K+D_2$ is effective. But
 $$
  K+D_2=\frac{\chi_1-m_1}{m_1} D_1 - \frac{\chi_2}{m_2} D_2 - \frac{\chi_3+m_3}{m_3} D_3\, .
  $$
If $m_1\geq \chi_1$, then this is anti-effective, which is a contradiction. Hence we always have
 \begin{equation}\label{eqn:chi1}
 0< m_1<\chi_1.
 \end{equation}

\begin{lemma}\label{lem:1-a}
Let $\varepsilon= \big(D_1,D_2, D_3\big)$, $\varepsilon'= \big(D'_1,D_2', D'_3\big)$ be two bases.
If $D_1'=\lambda_1 D_1$, then $[\varepsilon]=[\varepsilon']$. In particular, $D_2=D_2'$ and $D_3=D_3'$.
\end{lemma}

\begin{proof}
We restrict to $V=\la D_1\ra^\perp\subset H_2(Y,\RR)$, which is a vector space with a (negative) definite scalar product.
If $D_2=\lambda_2 D_2'$ (up to reordering), then it must be $D_3=\lambda_3 D_3'$ and $[\varepsilon]=[\varepsilon']$.
If $D_2,D_3$ are not proportional to $D_2',D_3'$, then $D_i\cdot D_j'\geq0$ for $i,j\in \{2,3\}$. 
If we take coordinates on $V$ so that $\{D_2, D_3\}$ is the standard basis, then 
$D_l' = \sum -a_{jl} D_j$ with $a_{jl}\geq 0$.
This is impossible since the first is effective and the second anti-effective.
\end{proof}

\begin{lemma}\label{lem:1-b}
Let $\varepsilon= \big(D_1,D_2, D_3\big)$, $\varepsilon'= \big(D'_1,D_2', D'_3\big)$,
$\varepsilon''= \big(D''_1,D_2'', D''_3\big)$ be three bases.
If $D_3, D_3',D_3''$ are proportional, then two of the bases are proj-equivalent.
\end{lemma}

\begin{proof}
Let $W=\la D_1,D_2\ra$, which is a vector space of dimension $2$ and signature $(1,1)$. Take an
orthonormal basis $\{e_1,e_2\}$ with $e_1=D_1/\sqrt{m_1}$, $e_2=D_2/\sqrt{m_2}$. 
If either $D_1'=D_1$ or $D_2'=D_2$ then $[\varepsilon']=[\varepsilon]$. 
Otherwise $D_1'\cdot D_1, D_1'\cdot D_2,D_2'\cdot D_1,D_2'\cdot D_2 \geq 0$.
In the above basis $D_1'=(a_1,-b_1), D_2'=(a_2,-b_2)$, with $a_j,b_j\geq 0$.
As they are orthogonal, $a_1a_2-b_1b_2=0$, hence $D_2'=\mu (b_1,-a_1)$ with $\mu>0$.
In an analogous manner, $D_1''=(c_1,-d_1), D_2''=\mu' (d_1,-c_1)$, with $c_1,d_1\geq 0$, $\mu'>0$.
Then $D_1'\cdot D_1''=a_1c_1-b_1d_1\geq 0$
and $D_2'\cdot D_2''=\mu\mu'(b_1d_1-a_1c_1)\geq 0$. So it must be $D_1'' \perp D_1'$, and hence
$D_1''$ is proportional to $D_2'$.
\end{proof}

\begin{definition}\label{def:good}
We call a curve $D_i$ \emph{good} if it does not pass through any singular point. 
We call a basis  $\varepsilon=( D_1,D_2,D_3)$ 
{good} if the three curves $D_1,D_2,D_3$ are good. In this case $m_1=D_1^2$, $m_2=-D_2^2$ and $m_3=-D_3^2$
 are positive integers. Also $\chi_i=2g_i-2 \in \ZZ$, and their homology classes lie in $H_2(X-P,\ZZ) \cong H^2(X,\ZZ)$.
\end{definition}

Fix some $N_0>0$ to be determined later.
Now we focus on bases of curves with genera $\{g_n=9n^2+1,g_a=9a^2+1,10\}$ for  $1<a\leq N_0$, 
and $N_0<n \leq N$. For each $a\in [2,N_0]$, we take a primer number $n_a\in [N_0+1,N]$. In particular $10<g_a<g_{n_a}$.
We require the numbers $n_a\neq n_{a'}$ for $a\neq a'$. 
We say that a number $a$ is \emph{bad} if the basis $(D_1,D_2,D_3)$ of genera $\{g_{n_a},g_a,10\}$ is not good.

 \begin{proposition} \label{prop:some-a}
 Let $\cota_0$ given in Corollary \ref{cor:Pbounded}. Then at most there are $4\cota_0$ bad numbers $a$.
 \end{proposition}
 
\begin{proof}
  By Corollary \ref{cor:Pbounded}, the number of 
  orbifold points is $\# P\leq \cota_0$. At an orbifold point,
  there are at most $2$ (nice) curves through it. 
  Let $\varepsilon_{{n_a}a}$ be the basis associated to the genera $\{g_{n_a},g_a,10\}$. For any bad $a$, 
  $\varepsilon_{n_aa}$ contains a curve through a point of $P$.
  Let us see that a point $p\in P$ can be at most in $4$ bases $\varepsilon_{n_aa}$. Therefore the
  number of bad numbers $a$ is $\leq 4\cota_0$.
  
 To check the assertion, fix a point $p\in P$ and 
 suppose that $a_1,a_2,a_3,a_4,a_5$ are bad with curves through $p$. 
 Let $\varepsilon_{n_{a_i}a_i}=(D_1^{(i)},D_2^{(i)},D_3^{(i)})$, $1\leq i\leq 5$.
 Note that $g_{a_i}\neq g_{a_j}$ and $g_{n_{a_i}}\neq g_{n_{a_j}}$ for $i\neq j$, since 
 $n_{a_i}\neq n_{a_j}$, and $g_{a_i}\neq g_{n_{a_k}}$ for $i,k$ since $a_i\leq N_0<n_{a_k}$.
By Lemma \ref{lem:1-b}, there must be three different curves among $D_3^{(i)}$. Reordering
we can suppose this for $i=1,2,3$. Then all curves in $\varepsilon_{n_{a_i}a_i}$, $i=1,2,3$ are different.
But it cannot be more than two curves through $p$, a contradiction.
\end{proof}

Taking $N_0=4\cota_0+1$, this guarantees the existence of some $a$ which is not bad.
Let now
 \begin{equation}\label{eqn:boundN}
 N_1=\max \left( n_a \, | \, a\in [2,N_0] \right).
 \end{equation}
This is a universal quantity, i.e.\ independent of $Y$.

\section{Universal geometric bounds}\label{sec:6}

Now we want to get universal bounds on some geometric quantities associated to an 
orbifold $Y$ satisfying Conditions \ref{condi}. As before, let $\pi: \tilde Y\to Y$ be the minimal resolution of singularities,
and let $\tilde K$ and $K$ be the canonical divisors of $\tilde Y$ and $Y$, respectively.
In this section, we use $N_1$ from 
(\ref{eqn:boundN}). 

\begin{lemma} \label{lem:boundK2-above}
There is a universal $\cota_1$ so that $\tilde K^2\leq K^2 \leq \cota_1$.
\end{lemma}

\begin{proof}
The first equality follows from (\ref{eqn:tildeK}). Next, by Proposition \ref{prop:some-a}, there
is some $a\leq N_0$ which is not bad. This means that the curves in the basis 
$\varepsilon_{n_aa}=(D_1,D_2,D_3)$, with genera
$\{g_{n_a},g_a,10\}$, do not pass through singular points. As $n_a \leq N_1$, we have $g_{n_a}=9n_a^2+1\leq 9N_1^2+1$.
By using (\ref{eqn:K}), we have 
 $$
K^2=\frac{(2g_{n_a}-2-m_1)^2}{m_1}- \frac{(2g_a-2+m_2)^2}{m_2} - \frac{(18+m_3)^2}{m_3} \leq \frac{(2g_{n_a}-2-m_1)^2}{m_1} \, ,
 $$
where we have assumed that the curve $D_1$ is the positive one (the other cases can be done similarly).
As $0<m_1<2g_{n_a}-2$ by (\ref{eqn:chi1}), we can compute the maximum value of the 
expression above to be $(2g_{n_a}-3)^2\leq (18N_1^2-1)^2$. So $K^2$ is universally bounded.
\end{proof}

\begin{lemma}\label{lem:boundK2-below}
There is a universal $\cota_2$ so that $K^2\geq \tilde K^2 \geq -\cota_2$.
\end{lemma}

\begin{proof}
We already know that $K^2\geq \tilde K^2$. Now take $a\leq N_0$ which is not bad,
and let $\{g_{n_a},g_a,10\}$ be the genera of a good basis of curves $(D_1,D_2,D_3)$. As they
do not pass through singular points, we denote the proper transforms under the 
resolution map $\pi:\tilde Y\to Y$ by the same letters $D_1,D_2,D_3$.
Recall that we denote $C_j$ the exceptional divisors.

To bound $\tilde K^2$, we note that 
$H^0(\tilde K+D_1) \cong H^0(D_1, K_{D_1})=\CC^{g_{n_a}}$. We assume that $D_1$ is the positive curve,
the other cases are similar. Write the linear system
 $|\tilde K+D_1|= Z+|F|$, where $Z$ is the base-point locus and $F$ is a free divisor. Then $Z\cdot D_1=0$,
 since $H^0(D_1,K_{D_1})$ is base-point free.
 Write the divisor $Z=T+\sum a_i D_i+\sum b_j C_j$, for $a_i\geq 0, b_j\geq 0$, and $T\geq 0$ not containing $D_i$ and $C_j$. 
 As $Z\cdot D_1=0$, we have
 $a_1=0$ and $T\cdot D_1=0$. In the rational equivalence class, we have 
 $T\equiv \sum \alpha_i D_i+\sum \beta_j C_j$. Again $\alpha_1=0$ and $\alpha_i\leq 0$ because $T\cdot D_i \geq 0$, $i=2,3$.
 Also $T\cdot C_j\geq 0$ for all $j$, implies that $\beta_j\leq 0$ for all $j$. This implies that $T$ is anti-effective and
 effective, hence $T=0$. Thus $Z=\sum a_i D_i+\sum b_j C_j$, hence $\tilde K \cdot Z\geq 0$ because 
 $\tilde K \cdot D_i\geq 0$ and $\tilde K \cdot C_j\geq 0$. Next
  $$
  (\tilde K+D_1 -Z)^2 =F^2 \geq 0.
  $$
So $ (\tilde K+D_1)^2 -2  (\tilde K+D_1)\cdot Z+Z^2 \geq 0$, and hence $ (\tilde K+D_1)^2 \geq 0$.
This reads $\tilde K^2 + 2(2g_{n_a}-2-m_1)+m_1 \geq 0$, whence 
$\tilde K^2\geq 4-4g_{n_a}+m_1\geq 5-4g_{n_a} \geq 1-36N_1^2$, using that $g_{n_a}=9n_a^2+1\leq 9N_1^2+1$.
\end{proof}

\begin{lemma}
 There is a universal $\cota_3$ so that $e(\tilde Y) \leq \cota_3$. 
\end{lemma}

\begin{proof}
As $h^{1,1}=b_2$, we have that the geometric genus is $p_g=h^{2,0}=0$. Also $b_1=0$ implies that the irregularity
is $q=0$. So the holomorphic Euler characteristic is $\chi(\cO_{\tilde Y})=1-q+p_g=1$. By Noether formula,
$\tilde K^2+e(\tilde Y)=12\chi(\cO_{\tilde Y})=12$, hence $e(\tilde Y)=12- \tilde K^2 \leq 12+\cota_2=\cota_3$.
\end{proof}

\begin{proposition}\label{prop:boundC2}
There is a universal $\cota_4$ so that if $C$ is a nice curve 
with $C^2=-m<0$ and genus $g=g(C)\geq 1$, then $m \leq 2g+\cota_4$.
\end{proposition}

\begin{proof}
 We apply  \cite[Theorem 10.14]{Ast} to the smooth variety $\tilde Y$. First we check that $\tilde K+\tilde C$ is effective, 
 which follows as in Lemma \ref{lem:Sigma}. If we have that $\tilde K+\tilde C$ is nef, then 
  \cite[Theorem 10.14]{Ast}  says that 
   \begin{equation}\label{eqn:chi-chi}
(\tilde K+\tilde C)^2\leq 3 e(\tilde Y-\tilde C) = 3e(\tilde Y)+6g-6.
   \end{equation}
To check that $\tilde K+\tilde C$ is nef, let $A$ be an effective curve. 
If $A=\tilde C$ then $(\tilde K+\tilde C)\cdot \tilde C=2g-2\geq 0$. So suppose $A\neq \tilde C$. If $\tilde K\cdot A\geq0$
then $(\tilde K+\tilde C)\cdot A \geq 0$. Also if $A^2\geq 0$ then write for an effective $\Sigma \equiv \tilde K+\tilde C$,
$\Sigma=rA+T$, $r\geq 0$, $T$ not containing $A$, and thus $(\tilde K+\tilde C)\cdot A =rA^2+T\cdot A\geq 0$.

So we are left with $\tilde K\cdot A<0$ and $A^2<0$. Then $A$ is a $(-1)$-curve. If $A\cdot \tilde C\geq 1$ then we
are again done. So also $A\cap \tilde C=\emptyset$.
Blow-down $A$ and let $\tilde Y\to \bar Y$ be the blow-down map. We can assume inductively that
in $\bar Y$ we have
  $(\bar K+\tilde C)^2\leq 3e(\bar Y)+6g-6$. So 
  $(\tilde K+\tilde C)^2-1\leq 3e(\tilde Y)-3+6g-6$, and (\ref{eqn:chi-chi}) follows.
  
  Now from (\ref{eqn:chi-chi}), $\tilde K^2+2\tilde K\cdot \tilde C+\tilde C^2 \leq 3 e(\tilde Y)+6g-6$, which reads
  $12-e(\tilde Y)+4g-4 - \tilde C^2 \leq 3e(\tilde Y)+6g-6$. Therefore
   $$
    -\tilde C^2\leq 4e(\tilde Y)+2g-14 \leq 4\cota_3+2g-14=2g+\cota_4\, ,
  $$
with $\cota_4=4\cota_3-14$.
Finally, Lemma \ref{lem:lem} says that for $C=\pi(\tilde C)$, then $C^2\geq \tilde C^2 $, so $-C^2 \leq -\tilde C^2 \leq 2g+\cota_4$.
\end{proof}

\section{Proof of the non-Sasakian property}\label{sec:7}

Our final purpose is to complete the main result (Theorem \ref{thm:main}).

\begin{theorem}\label{thm:main-Sasakian}
There is some $N$ large enough such that 
the K-contact manifold $M$ from Theorem \ref{thm:K-contact} does not admit
a Sasakian structure.
\end{theorem}

If $M$ admits a Sasakian structure, then it also admits a quasi-regular Sasakian structure. Therefore
there is a Seifert bundle $\pi:M \to Y$, where $Y$ is a K\"ahler  cyclic orbifold. From the homology of $M$
given by Theorem \ref{thm:K-contact}, we have that $b_1(Y)=0$, $b_2(Y)=3$ and the ramification locus
is given by a collection of curves 
 $$
 \varepsilon^{nm}=(D_1^{nm}, D_2^{nm}, D_3^{nm}),
 $$
which satisfy that $D_1^{nm}, D_2^{nm}, D_3^{nm}$ are disjoint and span $H_2(Y,\QQ)$, for each $n,m$.
They can coincide or intersect for different values of $(n,m)$. The genera of $D_1^{nm}$, $D_2^{nm}$, $D_3^{nm}$
are $\{9n^2+1,9m^2+1,10\}$, respectively.

We start with the collection of bases $\varepsilon_n=(D_1^n,D_2^n,D_3^n)$ associated
to $m=2$, $n\in [3,N]$. The genera of the curves are $\{g_n=9n^2+1, 37, 10\}$ with
$g_n>37$.

Recall the bound $\#P \leq  \cota_0$ from Corollary \ref{cor:Pbounded}. Then there are at most 
$2\cota_0$ curves among $D_1^n,D_2^n, D_3^n$ passing through points of $P$. All the curves
$D_1^n$ are distinct, but there can be repetitions among $D_2^n,D_3^n$. By (\ref{eqn:chi1}), if
$D_1^n$ is the positive curve then we have
$m_1<\chi_1=2g_n-2=18n^2$.

\begin{proposition}\label{prop:nandn}
There are some (universal) $n_0>0$ and positive integer $R$ 
and $N>n_0$ so that there exist two prime numbers $n,n'\in [n_0+1, N]$ with
 $$
 R\left( \frac{n^4}{m_1} - \frac{n'^4}{m'_1} \right) \in \ZZ,
 $$ 
where $m_1=(D_1^n)^2$, $m_1'=(D_1^{n'})^2 \in \ZZ$, and $0<m_1<18n^2$, $0<m_1'<18n'^2$.

We can select $n,n'$ from a previously given infinite collection of  primes $\cP \subset \ZZ_{>0}$. 
Only $N$ depends on $\cP$, otherwise it is universal. 
 \end{proposition}
 
 \begin{proof}
 Divide the set $[3,N]\cap \cP$ into classes $\cA_1,\ldots, \cA_l$ according to proj-equivalence of
   the basis $\varepsilon_n$, that is $[\varepsilon_n]=[\varepsilon_m]$ if
   and only if $n,m\in \cA_i$ for some $i$.
  It may happen that $D_2^n =D_2^m$ if $n\in \cA_i$, $m\in \cA_j$, $i\neq j$,
  but it cannot be for three different classes, by Lemma \ref{lem:1-b}. 
  If this happen for $\cA_i,\cA_j$, we  retain $\cA_i$ and discard
  $\cA_j$ so that $\# \cA_i\geq \# \cA_j$. Let $\cA_{i_1},\ldots, \cA_{i_t}$ be the
retained classes, and note that $\#( \cup \cA_{i_k}) \geq \cota/2$, where $\cota=\# ([3,N]\cap \cP)$. Repeat the same 
process with the curves $D_3^n$. The remaining classes  
$\cA_{j_1},\ldots, \cA_{j_s}$ have cardinality $\# (\cup \cA_{j_k})\geq \cota /4$.
Then two bases in $\cup \cA_{j_k}$ are either proj-equivalent, or their curves are all distinct. 

If a class $\cA_{j_k}$ contains two primes $n,n'>n_0$ ($n_0$ will be chosen later),
then let $\varepsilon_n=(D_1,D_2,D_3)$,
$\varepsilon_{n'}=(D_1',D_2',D_3')$ be the proj-equivalent bases. As $D_2=D_2'$, $D_3=D_3'$,
then $D_1\cdot D_1'\geq 0$, so $D_1,D_1'$ are the positive curves. We compute
 \begin{align*}
K^2 &=\frac{(\chi_1-m_1)^2}{m_1}- \frac{(\chi_2+m_2)^2}{m_2} - \frac{(\chi_3+m_3)^2}{m_3}  \\
   &=\frac{(\chi_1'-m'_1)^2}{m'_1}- \frac{(\chi_2+m_2)^2}{m_2} - \frac{(\chi_3+m_3)^2}{m_3}  \, ,
\end{align*}
with the usual meaning for $\chi_j, m_j$ and $\chi_j',m_j'$. Then
 $$
 \frac{(\chi_1-m_1)^2}{m_1}- \frac{(\chi_1'-m'_1)^2}{m'_1}=0.
 $$
 
Now suppose that one $\cA_{j_k}$ contains $2\cota_0+2$ primes $n>n_0$. At most $2\cota_0$ of the curves
 $D_1^n$ are not good. So there are two primes $n,n'>n_0$ associated to good curves, and hence
 using that $2g_n-2=18n^2$, we have
 $$
 \frac{(18n^2-m_1)^2}{m_1}- \frac{(18n'^2-m'_1)^2}{m'_1}=0,
 $$
with $m_1,m_1'$ integers. This gives the result (actually with $R=18^2$).  
 
Now suppose that all classes $\cA_{j_k}$ contain at most $2\cota_0+1$ primes $n>n_0$. 
Take $N>0$ so that in $[n_0+1,N]\cap (\cup \cA_{j_k})$ there are more than $(2\cota_0+1)2\cota_0+2$ primes.
This can be arranged if 
 \begin{equation}\label{eqn:TT}
 \cota /4-(n_0-2) \geq (2\cota_0+1)2\cota_0+2.
 \end{equation}
 Choose $N$ large enough so that $\cota$ is large enough for (\ref{eqn:TT}) to hold.
Now remove all classes $\cA_{j_k}$ that contain a curve which is not good. At most
there are $2\cota_0$ of them. Therefore there must be two primes $n,n'$ still left after this.
In that case, the bases $(D_1,D_2,D_3),(D_1',D_2',D_3')$ are both good, and in different classes.

Let us see first that $D_1,D_1'$ are the positive curves. Suppose 
for instance that $D_2$ is the positive curve.
Then
$$
K^2= \frac{(72+m_2)^2}{m_2} - \frac{(2g_n-2+m_1)^2}{m_1} - \frac{(18+m_3)^2}{m_3}  \, .
 $$
The first and last term are bounded by (\ref{eqn:chi1}) and Proposition \ref{prop:boundC2}. So 
 $$
 K^2 \leq \cota_5-  \frac{(2g_n-2+m_1)^2}{m_1},
 $$
for some universal $\cota_5$. This implies the bound $K^2 \leq \cota_5-(8g_n-8)$. By
Lemma \ref{lem:boundK2-below}, $-\cota_2\leq \cota_5-8g_n+8$ and so $g_n=9n^2+1 \leq 1+\frac18(\cota_5+\cota_2)$. 
This means that there is $n_0=[\frac1{72}(\cota_5+\cota_2)]+1$ such that for $n\geq n_0$, $D_1=D_1^n$ is the positive curve.
This $n_0$ is universal.

Now take $n,n'\geq n_0+1$. Then $D_1,D_1'$ are positive curves, we have 
\begin{align*}
 K^2&= \frac{(2g_n-2-m_1)^2}{m_1} -\frac{(72+m_2)^2}{m_2}- \frac{(18+m_3)^2}{m_3} \\
  &= \frac{(2g_{n'}-2-m'_1)^2}{m'_1} -\frac{(72+m_2')^2}{m_2'}- \frac{(18+m'_3)^2}{m'_3} \, .
\end{align*}
Recalling that $2g_n-2=18n^2$, we  have
 $$
  \frac{18^2n^4}{m_1} - \frac{72^2}{m_2}- \frac{18^2}{m_3} 
  - \frac{18^2n'^4}{m'_1} + \frac{72^2}{m_2'}+\frac{18^2}{m'_3} \in \ZZ.
 $$
where $0<m_1<18n^2$, $0<m_1'<18n'^2$. By Proposition \ref{prop:boundC2}, $m_2\leq \cota_4+74$ and 
$m_3\leq \cota_4+20$. Then take $R=18^2\cdot \lcm (2,3,4,\ldots,  \cota_4+74)$, and we get the statement.

The number $N$ has to be chosen large enough so that $\cota$ satisfies the inequality
(\ref{eqn:TT}). It depends on $\cP$ clearly.
\end{proof}

Now take $n,n'>n_0$ prime numbers satisfying the condition in Proposition \ref{prop:nandn}.
Take $d=\gcd(m_1,m_1')$ and write $m_1=d a$, $m_1'=d a'$, with $\gcd(a,a')=1$. Then
 $\frac{n^4R}{a} - \frac{n'^4R}{a'}$ is an integer,
 from where $a|n^4R$ and $a'|n'^4R$. Given that $a< 18n^2$ and
 $a'< 18n'^2$, there is a finite set of possibilities for $a,a'$. 
 Let $D=\{d_1,\ldots, d_t\}$ be the divisors of $R$. Then $a\in \{d_i, d_in, d_in^2\}$, and
 $a'\in \{d_i, d_in', d_in'^2\}$. Therefore 
 \begin{equation}\label{eqn:quotient}
  \frac{m_1}{m_1'} =\frac{d_i}{d_j} \frac{n^\beta}{n'^\gamma}
 \end{equation}
with $\beta,\gamma=0,1,2$, $d_i,d_j\in D$.

Next $K^2$ is bounded by Lemma \ref{lem:boundK2-above}, hence 
 $$
 \frac{(18n^2-m_1)^2}{m_1} \leq \cota_6\, ,
 $$
for some universal $\cota_6$, using also Proposition \ref{prop:boundC2} to bound $m_2,m_3$. Then $m_1$ lies in the interval
 $$
 m_1 \in \left[ 18n^2+\frac{\cota_6}{2}-\sqrt{18n^2\cota_6 +\frac{\cota_6^2}{4}} , 18n^2\right).
 $$
In particular, 
 \begin{equation}\label{eqn:corr02}
18n^2-\sqrt{18\cota_6} n \leq m_1 <18n^2,
\end{equation}
and analogously for $m_1'$. Now 
 \begin{equation}\label{eqn:m1m1}
 \frac{m_1}{m'_1} \in \left( \frac{18n^2-\sqrt{18\cota_6} n}{18n'^2},
 \frac{18n^2}{18n'^2-\sqrt{18\cota_6} n'} \right).
 \end{equation}
 
Consider the set ${\mathcal R}=\{s=\frac{d_i}{d_j} \,| d_i\in D\}$. Let $\epsilon=\min (|1-s| \, |\, s \in {\mathcal R} , s\neq 1)>0$. This
is a universal number. 
Enlarging $n_0$, we have that
for primes $n,n'\geq n_0+1$, the quotient (\ref{eqn:m1m1}) 
is within $\epsilon$ of $\frac{n^2}{n'^2}$, i.e.\ in the interval
 \begin{equation}\label{eqn:n2n2}
 \left( (1-\epsilon) \frac{n^2}{n'^2},(1+\epsilon) \frac{n^2}{n'^2}\right).
 \end{equation}
 This $n_0$ is again universal (depends on $R$ and $\cota_6$).
 
We choose our collection of primes $\cP=\{n_1,n_2,\ldots \}$ 
in Proposition \ref{prop:nandn} in increasing order as follows. First choose 
$n_0\geq R(1+\epsilon)$, so that $n_i> n_0 \geq R(1+\epsilon)$.
Next take $n_{i+1}>(1-\epsilon)^{-1} R n_i^2$, for $i\geq 1$.

Now given $n=n_i,n'=n_j$, $i>j$, then all numbers (\ref{eqn:quotient}) are away from (\ref{eqn:n2n2}).
This is proved as follows: first all quotients $s=\frac{d_i}{d_j}\in [\frac1R,R]$. Next, 
$(1-\epsilon)\frac{n^2}{n'^2}\geq n R$, which is bigger than any of
the expressions $s, s n, s\frac1{n'}, s\frac{n}{n'}, s\frac1{n'^2}, s\frac{n}{n'^2}$. Also
$(1+\epsilon)\frac{n^2}{n'^2}\leq \frac1R \frac{n^2}{n'}$, which is smaller than any of the
expressions $s\frac{n^2}{n'}, s n^2$. 
Hence it must be 
 $$ 
 \frac{m_1}{m'_1}=\frac{n^2}{n'^2}\, ,
 $$
since $s=\frac{d_i}{d_j}\notin (1-\epsilon,1+\epsilon)$  unless $s=1$. 
Therefore $m_1=d_i n^2$, $m_1'=d_i n'^2$, for some $d_i\in D$. By (\ref{eqn:corr02}), this is impossible.

This contradiction shows that for such $N$ in Proposition \ref{prop:nandn}, Theorem \ref{thm:main-Sasakian} holds.

\begin{remark}
All the numbers $\cota_0,\cota_1,\ldots, \cota_6,n_0, R, N_0, N_1$ and $N$ that have appeared along the proof can be
determined. So $N$ in Theorem \ref{thm:main-Sasakian} can be found explicitly.
\end{remark}

\end{document}